\documentclass[12 pt, leqno]{article}
\usepackage{preamble}

\title{\textsc{Pretorsion theories in prenormal categories}}
\author{Sandra Mantovani\textsuperscript{*}
\\
\texttt{sandra.mantovani@unimi.it}
\and 
Mariano Messora\textsuperscript{*}
\\
\texttt{mariano.messora@unimi.it}
\\
}
\date{
\footnotesize{\textsuperscript{*}Department of Mathematics, University of Milan, Via Cesare Saldini 50, 20133 Milan, Italy}%
\vspace{-2em}
}%

\begin{document}%
\maketitle%

\hrule\vspace{.7em}%
\thispagestyle{empty}%
\noindent\textbf{Abstract.} In this paper we extend several classical results on pointed torsion theories --  also known as torsion pairs -- to the setting of non-pointed torsion theories defined via kernels and cokernels relative to a fixed class of trivial objects (often referred to as \emph{pretorsion theories}). Our results are developed in the recently introduced framework of (non-pointed) \emph{prenormal categories} and other related contexts. Within these settings, we recover some characterisations of torsion and torsion-free subcategories, as well as the classical correspondences between torsion theories and closure operators. We also suitably extend a correspondence between torsion theories and (stable) factorisation systems on the ambient category, known in the homological case. Some of these results are then further specialised to an appropriate notion of hereditary torsion theory. Finally, we apply the developed theory to construct new examples of pretorsion theories.
\vspace{.7em}
\hrule
\vspace{.7em}
\vspace{.7em}
\hrule
\renewcommand{\baselinestretch}{0.8}\normalsize
\tableofcontents
\renewcommand{\baselinestretch}{1}\normalsize
\vspace{.7em}
\bigskip
\hrule
\vspace{.5em}
\noindent \href{https://doi.org/10.1016/j.jpaa.2026.108239}{Published in \emph{J.\ Pure Appl.\ Algebra} 230.4, 108239 (2026), pp.~1--26}
\restoregeometry
\section*{Introduction}
\addcontentsline{toc}{section}{Introduction}
The classical notion of torsion theory (\cite{DICKSON}), originally introduced for abelian categories, has since been studied by many authors in a variety of non-abelian contexts (see, for instance, \cite{BOURN,THOLEN,JANELIDZE07,EVERAERT}). More recently, it was shown that the notion can be  extended even beyond the setting of \emph{pointed} categories (that is, categories with a zero object) by replacing the zero object with a fixed, chosen class $\tobj$ of \emph{trivial objects}, and by working with kernels and cokernels defined relative to this class. This broader framework, along with several generalisations of it, has been studied in various works in the literature, including \cite{MARKI13,GRANDIS20,PRETORSION,HTT}, with a range of different perspectives, assumptions and terminology. In particular, the notion of \emph{pretorsion theory} introduced in \cite{PRETORSION} fits within this general framework, but with the class $\tobj$ determined only afterwards as the intersection of the torsion and torsion-free classes. In this paper, by contrast, we adopt the term \emph{\ztt{}} in order to emphasise the role of a class $\tobj$ of trivial objects fixed in advance. Regardless of the terminology, $\tobj$-torsion theories provide a far-reaching generalisation of the original notion of torsion theory, with a wide variety of examples found throughout the literature.

Despite the recent surge of interest in this notion, its general theoretical background remains at a rather incomplete stage of development. In the literature on pointed torsion theories, a number of now classical results have been established within various categorical settings. 
Such results include, for example, characterisations of torsion and torsion-free subcategories in terms of special properties of (co)reflections, as well as results describing the close connections between torsion theories and closure operators, or between torsion theories and factorisation systems (\cite{BOURN,THOLEN,JANELIDZE07,ROSICKY,EVERAERT}).
By contrast, analogous results for \ztts{} are still comparatively underdeveloped. 

The aim of this paper is to identify a meaningful categorical context where some of these classical results on torsion theories can be suitably extended to the setting of \ztts{}.

Such task is not straightforward. A key difficulty is that the more powerful of these results rely on  strong structural assumptions on the ambient category, such as regularity, protomodularity, or other related conditions. 
These are known to endow pointed kernels and cokernels with valuable categorical properties that are systematically exploited in the proofs of the aforementioned classical results: for instance, in a \emph{homological} category --  i.e.\ a pointed regular protomodular category --  normal and regular epimorphisms coincide and are pullback-stable (in other words, homological categories are \emph{normal} in the sense of \cite{NORMAL}). 
By contrast, these same assumptions on the ambient category fail to interact with the `relative' kernels and cokernels in any meaningful way: `$\tobj$-normal epimorphisms' (i.e. morphisms underlying cokernels relative to $\tobj$) are in general not regular epimorphisms, and protomodularity certainly does not force regular epimorphisms to become $\tobj$-normal. As a result, extending the stronger classical results on torsion theories to the relative setting proves to be a rather delicate matter. 



Nevertheless, some progress toward this goal has been made. For instance, the aforementioned work \cite{PRETORSION} extends to \ztts{} the classical characterisation of torsion-free subcategories as epi-reflective subcategories whose associated radical is idempotent. More recently, \cite{ANDREAS} establishes a generalisation of results from \cite{EVERAERT} to the context of \ztts{}, though their approach requires imposing rather restrictive conditions on the subcategory of trivial objects (to the point that a subcategory satisfying such conditions, if it exists, is unique): while not required to be reflective, it must be posetal, mono-coreflective, and such that the coreflector inverts monomorphisms. In addition, the ambient category in \cite{ANDREAS} is still assumed to be both regular and protomodular.

The perspective of the present paper is quite different. Rather than assuming regularity or protomodularity, we look for a general categorical framework that independently equips $\tobj$-normal epimorphisms with some of the essential features their pointed counterparts enjoy in the homological context. 
Such a framework should be 
robust enough to recover, at least in part, the classical results on torsion theories, (co)reflections, factorisation systems, and closure operators,
 while remaining
sufficiently general to encompass many of the example of \ztts{} found in the literature, and possibly produce new ones.

We find such a framework in the setting of \emph{prenormal categories} and certain weakenings of this notion. Prenormal categories were introduced recently in \cite{PN}; their definition is modelled on that of regular categories, but with the roles of coequalisers and kernel pairs replaced by cokernels and kernels -- possibly relative to a fixed class $\tobj$, in which case one speaks of $\tobj$-prenormal categories. Although \normal{} categories form a broad generalisation of homological and normal categories (even in the pointed setting), they still retain many fundamental properties familiar from algebra. Notably, they admit a stable factorisation system based on (relative) normal epimorphisms, as well as a categorical analogue of Noether’s third isomorphism theorem. These properties, in turn, make it possible to partially recover and extend to \ztts{} some of the classical results discussed earlier, often by retaining the same overall structure of the proofs -- though additional technicalities inevitably arise. This is carried out in \zcref{sec-T-and-TF,sec-tt-and-CO,sec-TT-and-FS} of the present paper, following a general review of relative (co)kernels, pretorsion theories and prenormal categories in \zcref{sec-preliminaries,sec-context}. In particular, we will focus on how the following key results on torsion theories extend to the $\tobj$-setting: characterisations of torsion and torsion-free subcategories, the classical correspondence between torsion theories and closure operators, and the correspondences between torsion theories and factorisation systems. \zcref{sec-hereditary} is devoted to specialising some of the preceding results to (the relevant notion
of) hereditary torsion theories. Finally, in \zcref{sec-examples} we look at examples of \ztts{} in \normal{} categories, some taken from the literature and others newly constructed here.
\section{Preliminaries}
\label{sec-preliminaries}
In this section, we review the background on kernels and cokernels defined relative to a distinguished class of trivial objects, as well as the resulting notion of torsion theory derived from them. 

We consider a category $\cat{C}$ and a class $\tobj$ of objects, identified with the full subcategory they span. These will remain fixed for the remainder of the paper.

\subsection*{Categories with a distinguished class of trivial objects}
For the theoretical background reviewed in this subsection, we refer the reader, for example, to \cite{GRANDIS13,GRANDIS92,MARKI13,PN}. 
\begin{definition}
\label{Z-def}
     A morphism in $\cat C$ is said to be \emph{\nptrivial{}} if it factors through an object in $\cat Z$.  
     
     Let $f\colon A\to B$ be a map in $\cat C$. A \emph{\npkernel{}} of $f$ is given by an object $K$ and a map $k\colon K\to A$ such that $\comp kf$ is \nptrivial{} and for any other map $x\colon X\to A$ such that $\comp xf $ is \nptrivial{}, there exists a unique $x'\colon X\to K$ such that $\comp{x'}k=x$.  We will sometimes refer to the kernel $(K, k)$ simply as $K$ or $k$, when the distinction is clear from context.
     Of course, \emph{\npcokernel{}s} are defined dually.
     
     A \emph{\npexact{} sequence} is given by a pair of composable maps 
\(
\begin{tikzcd}[cramped, sep=1.5em]
A\arrow[r, "f"]&B\arrow[r,"g"]&C
\end{tikzcd}
\)
    such that $(A,f)$ is the \npkernel{} of $g$ and $(C,g)$ is the \npcokernel{} of $f$.
\end{definition}
\npkernel{}s share many familiar properties with ordinary kernels, as highlighted in the following proposition.
\begin{proposition}
\label{pb-nmono}
    Let $(A,f)$ be a \npkernel{} of a map $g\colon B\to C$. Then the following hold.
    \begin{enumerate}
        \item $f$ is a monomorphism;
        \item\label{normmono-is-tker} if $h$ is a map in $\cat C$ such that $\comp hf$ is \nptrivial{}, then $h$ itself is \nptrivial{}.
        \item if $f$ admits a \npcokernel{} $(Q,q)$, then $(A,f)$ is the \npkernel{} of $q$;
        \item \label{specific-pb-mono}if $b\colon B'\to B$ is a map in $\cat C$ such that the pullback $f'\colon A'\to B'$ of $f$ along $b$ exists, then $(A',f')$ is the \npkernel{} of the composite $\comp bg$.
        \[
        \begin{tikzcd}
            A'\rar["f'"]\dar& B'\dar["b"]
            \\
            A\rar["f"']&B\rar["g"']&C
        \end{tikzcd}
        \]
    \end{enumerate}
\end{proposition}

    We will call \emph{\npnormmono{}s} the   monomorphisms in $\cat C$ underlying \npkernel{}s. We denote by $\NMono$ the class of \npnormmono{}s in $\cat C$, and we graphically represent them using the special arrow `$\nmonoto$'. By the previous proposition, the class of \npnormmono{}s is pullback stable. Dually, we call \emph{\npnormepi{}s} the   epimorphisms underlying \npcokernel{}s. We denote by $\NEpi$ the class of \npnormepi{}s, and we use  the special arrow `$\nepito$' to graphically represent them. 
    
    Note that \npnormmono{} are not, in general, regular monomorphisms, and \npnormepi{} are not regular epimorphisms (see \cite{PN}). 
    Moreover, in the general case, \npkernel{}s are not defined as limits. They can, however, be obtained as pullbacks when the subcategory $\cat Z$ is mono-coreflective. We have the following proposition.

\begin{proposition}
{
\renewcommand{\corefl}[1]{Z}
\renewcommand{\cuni}[1]{\epsilon}
    \label{monocoref} Given an object $B$ in $\cat C$ and a monomorphism $\cuni B\colon \corefl B\monoto B$ with $Z\in\tobj$, then $(\corefl B,\cuni B)$ is a coreflection of $B$ in $\tobj$ if and only if it is a \npkernel{} of $\id B$. 
    In this case, given maps 
    \(
    \begin{tikzcd}[cramped, sep=1.5em]
    K\arrow[r, "k"]&A\arrow[r,"f"]&B,
    \end{tikzcd}
    \)
    then $(K,k)$ is the \npkernel{} of $f$ if and only if there exists a (unique) map $h\colon K\to \corefl B$ such that the following square is a pullback.
    \begin{equation*}
    \begin{tikzcd}
    K\arrow[r, "h"]\arrow[d, "k"'] & \corefl B\arrow[d, "\cuni B"]
    \\
    A\arrow[r, "f"'] &B
    \end{tikzcd}
    \end{equation*}
}
\end{proposition}


Next, we turn to a central property of \npkernel{}s, which plays a crucial role in this paper. In the pointed setting, the kernel of the pullback of a given map is isomorphic to the kernel of the original map. In the relative context, instead, this property holds only in a significantly weakened form -- posing one of the main obstacles to the generalisations we aim to develop.
\begin{proposition}
    \label{kernel-of-parall}
    Consider the following pullback diagram where $(K,k)$ is the \npkernel{} of $f$.
    \[
    \begin{tikzcd}
        K\rar["k",normmono]&X\rar["f"]\dar["x"']&Y\dar["y"]
        \\
        &X'\rar["f'"']&Y'
    \end{tikzcd}
    \]
    If $y$ is a \npnormmono{}, then $(K,\comp kx)$ is the \npkernel{} of $f'$.
\end{proposition}
\begin{proof}
    Suppose $(Y,y)$ be the \npkernel{} of a map $b\colon Y'\to B$, and let $a\colon A\to X'$ be a map such that $\mcomp{a,f'}$ is \nptrivial{}. It follows that the composite $\mcomp{a,f',b}$ is \nptrivial{}, and hence there exists a map $a'\colon A\to Y$ such that $\mcomp{a,f'}=\mcomp{a',y}$. In turn, this yields a map $u\colon A\to X$ such that $\mcomp{u,f}=a'$ and $\mcomp{u,x}=a$. Since $\mcomp{u,f,y}=\mcomp{a,f'}$ is \nptrivial{}, by \zcref{pb-nmono} (\zcref*{normmono-is-tker}) we obtain that $\comp uf$ is \nptrivial{}, yielding a map $v\colon A\to K$ such that $\comp v k=u$.
    \[
    \begin{tikzcd}[ bezier bounding box]
        & A\arrow[dd, controls={+(-10em,0) and +(-10em,0)}, "a"']\arrow[dl, bend right, "v"' near end]\arrow[d, "u"']\arrow[dr, bend left, "a'"]
        \\
        K\rar["k"]&X\dar["x"']\rar["f"]&Y\dar["y"]
        \\
        &X'\rar["f'"']&Y'\rar["b"']& B
    \end{tikzcd}
    \]
     It immediately follows that $\mcomp{v,k,x}= a$ as required.
\end{proof}
\begin{remark}
\label{remark-pb}
   If \zcref{kernel-of-parall} held for arbitrary $y$, then the full subcategory $\tobj$ of trivial objects would necessarily be a groupoid.  Indeed, consider any map $y\colon Y\to Y'$ in $\tobj$. The following commutative square is a pullback.
    \[
    \begin{tikzcd}
        Y\rar[equal]\dar["y"']&Y\dar["y"]
        \\
        Y'\rar[equal] &Y'
    \end{tikzcd}
    \]
    Since $Y,Y'\in\tobj$, the \npkernel{} of $\id Y$ is $(Y,\id Y)$ and the \npkernel{} of $\id{Y'}$ is $(Y',\id{Y'})$.Therefore, if $\id Y $ and $\id {Y'}$ had the same \npkernel{}, the map $y$ would be an isomorphism. 
    
    Note that one can easily prove that a partial converse also holds: if $\tobj$ is a groupoid and a coreflective subcategory of $\cat C$, then \zcref{kernel-of-parall} holds for arbitrary $y$.
\end{remark}
Finally, we consider a class of morphisms that will play an important role in this paper. We characterise this class in the following proposition.
\begin{proposition}
  \label{tker-charact} Consider morphisms in $\cat C$
  \[
  \begin{tikzcd}
      A\rar["f", normmono] & B\rar["g"] & C,
  \end{tikzcd}
  \]
  with $(A,f)$ the \npkernel{} of $g$.   Then the following are equivalent.
\begin{enumerate*}[(i)]
    \item The map $f$ is \nptrivial{};
    \item $A$ is in $\cat Z$;
    \item $(A,f)$ is the \npkernel{} of an isomorphism;
    \item $(A,f)$ is a coreflection of $B$ in $\cat Z$; \item every map $x\colon X\to B$ is \nptrivial{} whenever $\comp xg$ is \nptrivial{}.
\end{enumerate*}
\end{proposition}
We denote by $\TKer$ the class of maps whose \npkernel{} exists and is \nptrivial{} (according to any of the equivalent statements in \zcref{tker-charact}), and we graphically represent a map with a \nptrivial{} \npkernel{} by the special arrow `\tkerto'. In the presence of \npkernel{}s, by \zcref{pb-nmono}, \zcref{normmono-is-tker}, we have that $\NMono\subseteq\TKer$.
\subsection*{Non-pointed torsion theories}
In this subsection we review the notion of torsion theory based on \npkernel{}s and \npcokernel{}, as just discussed. We omit the proofs of known results. For further details, see, for example, \cite{HTT,PRETORSION, GRANDIS20}.
\begin{definition}
\label{ztt}
    A \emph{\ztt{}} on $\cat C$ is a pair of full replete subcategories \ttt{} such that the following conditions hold.
    \begin{enumerate}
        \item Every morphism $f\colon A\to B$ in $\cat C$, with $A\in\TNC$ and $B\in\TFC$, is \nptrivial{}.
        \item For every object $X$ in $\cat C$, there exists a \npexact{} sequence
        \[
        \begin{tikzcd}
            A\rar &X\rar &B
        \end{tikzcd}
        \]
        with $A\in\TNC$ and $B\in\TFC$. Such a sequence is called a $(\TNC,\TFC)$-presentation of $X$.
    \end{enumerate}
\end{definition}
Notice that in \cite{PRETORSION} a \ztt{} \ttt{} is called a \emph{pretorsion theory}, and $\tobj$ is taken to be the intersection $\TNC\cap \TFC$. We adopt the name \ztt{} to emphasise the role of the subcategory of trivial objects.

For the remainder of the section, we will consider some initial properties and characterisations of the subcategories $\TNC$ and $\TFC$ underlying a \ztt{}, most of which are known. We will use the following terminology for a subcategory $\cat A \hookrightarrow\cat C$:
\begin{itemize}
    \item $\cat A$ is a \emph{\ztf{}} (respectively, \emph{\ztn{}}) subcategory of $\cat C$ if there exists a subcategory $\cat B$ of $\cat C$ such that $(\cat B, \cat A)$ (respectively $(\cat A,\cat B)$) is a \ztt{} on $\cat C$;
    \item $\cat A$ is a \emph{\zrefl{}} (respectively, \emph{\zcor{}}) subcategory of $\cat C$ if the inclusion $\cat A \hookrightarrow\cat C$ admits a left (respectively, right) adjoint and the components of the unit (respectively, of the counit) of the adjunction are \npnormepi{}s (respectively, \npnormmono{}s);
    \item $\cat A$ is \emph{closed under \zext{}s} if $X\in\cat A$ whenever there exists a \npexact{} sequence $A_1\to X\to A_2$ such that $A_1,A_2\in\cat A$.
\end{itemize}

\begin{proposition}
\label{TF-is-ref}
\label{char-tor-obj}
Let \ttt{} be a \ztt{} on $\cat C$.
\begin{enumerate}
    \item \ttt{}-presentations are unique up to unique isomorphism: given two \ttt-presentations $A\to X\to B$ and $A'\to X\to B'$ of the same object $X$, there exist unique isomorphisms $A\to A'$ and $B\to B'$ such that the following diagram is commutative.
   \begin{equation}
   \label{T-F-prez}
    \begin{tikzcd}[row sep=2em, column sep = 5em]
        A\rar\dar&X\rar\dar[equal]&B\dar
        \\
        A'\rar&X\rar&B'
    \end{tikzcd}
    \end{equation}
    \item The inclusion functor $\TFC\hookrightarrow\cat C$ admits a left adjoint $\TFF\colon \cat C\to\TFC$ and the inclusion functor $\TNC\hookrightarrow \cat C$ admits a right adjoint $\TNF\colon\cat C\to\TNC$. Calling $\TFU{}$ the unit of the first adjunction and $\TNU{}$ the counit of the latter, then for every object $X$ in $\cat C$, the following is a \ttt-presentation of $X$:
    \[
    \begin{tikzcd}
        \TNP X\rar["\TNU X"]&X\rar["\TFU X"]&\TFP X.
    \end{tikzcd}
    \]
    \item \label{TF-is-ref-3}Consider a \ttt-presentation of an object, like \ref{T-F-prez} above. Then the following are equivalent
    \begin{enumerate}[i)]
        \item $X\in\TFC$;
        \item $\TFU X$ is an isomorphism;
        \item $\TNP X\in\tobj$;
        \item $\TNU X$ is \nptrivial{};
        \item for all $Y\in\TNC$, all morphisms $f\colon Y\to X$ are \nptrivial{}.
    \end{enumerate}
    Of course, a dual characterisation holds for torsion objects.
\end{enumerate}
\end{proposition}
\begin{proposition}
\label{other-TF-properties}
    Let $\TFC$ be a \ztf{} subcategory of $\cat C$. The following properties hold.
    \begin{enumerate}
        \item $\TFC$ is a \zrefl{} subcategory of $\cat C$;
        \item If $f\colon X\to Y$ is a \npnormmono{} or a map with \nptrivial{} \npkernel{} and $Y\in\TFC$, then $X\in\TFC$;
        \item $\TFC$ is closed under \zext{}s in $\cat C$.
    \end{enumerate}
    \begin{proof}
        We just prove that if $f\colon X\tkerto Y$ has a \nptrivial{} \npkernel{} and $Y\in\TFC$, then $X\in\TFC$. Consider the following diagram, where $\TNU{}\colon\TNP{}\tto \id{\cat C}$ denotes the coreflection of the torsion-free subcategory.
        \[
        \begin{tikzcd}
        \TNP X\rar{\TNP f}\dar["\TNU X"']&\TNP Y\dar{\TNU Y}
        \\
        X\rar["f"',tker]&Y
        \end{tikzcd}
        \]
        As $Y\in\TFC$, the coreflection $\TNU Y$ is \nptrivial{}, and hence $\mcomp{\TNU X,f}=\mcomp{\TNP f,\TNU Y}$ is \nptrivial{}. Since $f$ has \nptrivial{} \npkernel{}, by \zcref{tker-charact} it follows that $\TNU X$ is \nptrivial{} and thus $X\in\TFC$. 
    \end{proof}
\end{proposition}
\begin{remark}
    \label{H-K}
    Let $\TFC$ be a \zrefl{} subcategory of $\cat C$, with reflector $\TFF$ and unit $\TFU{}$. Assume that the components of $\TFU{}$ admit \npkernel{}s. We can consider the two following full subcategories of $\cat C$.
    \begin{itemize}
        \item $\cat K = \{X\in\cat C\mid X\textnormal{ is a \npkernel{} of some component of }\TFU{}\}$
        \item $\begin{aligned}[t]
            \!\cat H&=\{X\in\cat C\mid \TFU X\textnormal{ is \nptrivial{}}\}
            \\
            & =\{X\in\cat C\mid \TFP X\in\tobj\}
            \\
            &=\{X\in\cat C\mid \textnormal{every map $f\colon X\to Y$ with $Y\in\TFC$ is \nptrivial{}}\}
        \end{aligned}$
    \end{itemize}
    We have that $\TFC$ is \ztf{}
 if and only if $\cat K=\cat H$ if and only if $\cat K\subseteq \cat H$. Indeed, if $\TFC$ is \ztf{}, then $\cat K=\cat H$ is the associated \ztn{} subcategory by the dual of \zcref{TF-is-ref}, \zcref{TF-is-ref-3}. Vice versa, if $\cat K\subseteq \cat H$, then $(\cat K,\TFC)$ is a \ztt{} because every object admits a $(\cat K,\TFC)$-presentation by definition of $\cat K$, and every morphism from an object in $\cat K$ to an object in $\TFC$ is \nptrivial{} by definition of $\cat H$.
\end{remark}
\begin{proposition}
\label{idmpt-rad}
    Let $\TFC$ be a \zrefl{} subcategory of $\cat C$ with reflector $\TFF$ and unit $\TFU{}$. Assume that, for all objects $X\in\cat C$, the map $\TFU X\colon X\to\TFP X$ admits a \npkernel{} $k_X\colon KX\to X$. Then $\TFC$ is \ztf{} if and only if $k_{KX}$ is an isomorphism for all $X\in\cat C$.
\end{proposition}
{
\newcommand{\KSC}{\cat K}
\begin{proposition}
    Let $\TFC$ be a \zrefl{} subcategory of $\cat C$, with $\TFP{}\colon \cat C\to\TFC$ the reflector and $\TFU{}$ the unit of the reflection. Let $\KSC$ denote the full subcategory of $\cat C$ consisting of those objects that are isomorphic to a \npkernel{} of some component of the unit $\TFU{}$. Then the following are equivalent.
    \begin{enumerate}[i)]
        \item \label{tf-K-1}$\TFC$ is \ztf{};
        \item\label{tf-K-2} $\TFC$ is closed under \zext{}s and $\KSC$ is closed under \znq s. 
    \end{enumerate}
    \begin{proof}
        The implication \zcref{tf-K-1}$\implies$\zcref{tf-K-2} follows from \zcref{other-TF-properties} and its dual.

        Conversely, suppose \zcref{tf-K-2} holds. By \zcref{H-K}, it is enough to prove that any $X\in\KSC$ has \nptrivial{} $\TFC$-reflection. Let $X$ be any object in $\KSC$.  Since $\KSC$ is closed under \znq{}s, the $\TFC$-reflection $\TFP X$ of $X$ is also in $\KSC$ and hence, there exists a \npexact{} sequence
        \[
        \begin{tikzcd}
            \TFP X\arrow[r, normmono] &Y \arrow[r,"\TFU Y",normepi] &\TFP Y
        \end{tikzcd}
        \]
        for some objects $Y\in\cat C$. As $\TFC$ is closed under \zext{}s, we deduce that $Y$ lies in $\TFC$. Therefore, $\TFU Y$ is an isomorphism and $\TFP X\in\tobj$ (see \zcref{tker-charact}).
    \end{proof}
\end{proposition}
}
\section{The context}
\label{sec-context}
In this section, we quickly review \emph{prenormal categories} and the closely related notion of \emph{semi-prenormal categories}. As mentioned in the introduction, the definition of prenormal category is based on that of a regular category, but shifts the focus from coequalisers and kernel pairs to (possibly non-pointed) kernels and cokernels. Semi-prenormal categories provide a weakening of this notion; their central condition (specifically, \zcref{wnorm-3}) was first studied by Grandis (\cite{GRANDIS13}), albeit with different aims and terminology.
We do not enter into detail here, but simply record the main properties needed for what follows. For further background, see \cite{PN}.
\begin{definition}
\label{wnorm}
Let $\cat C$ be a category and let $\tobj$ be a full, replete subcategory of $\cat C$. We say that $\cat C$ is \emph{\wnormal{}} if $\tobj$ is mono-coreflective in $\cat C$ and the following properties hold.
\begin{enumerate}
\zcsetup{reftype=condition}
    \item\label{wnorm-1} $\cat C$ admits pullbacks along \npnormmono{}s;
    \item\label{wnorm-2}  $\cat C$ admits \npcokernel{}s of \npkernel{}s;
    \item \label{wnorm-3}the pullback of a \npnormepi{} along a \npnormmono{} is a \npnormepi{}.
\end{enumerate}

We further say that $\cat C$ is \emph{$\tobj$-prenormal}, if it is finitely complete and \npnormepi{}s are stable under pullback along arbitrary morphisms.

When $\cat C$ is pointed and $\tobj$ is the class of zero objects, we refer to \mnormality{} simply as \mpnormality{}. 

\end{definition}
Before moving forward, we mention two simple pointed examples to give a more concrete sense of the above definitions. First of all, the category of commutative monoids is a regular category which is prenormal but not normal in the sense of \cite{NORMAL}, as regular and normal epimorphisms are both pullback-stable but they constitute different classes. The category of pointed sets is \wpnormal{} but not \pnormal{}, as normal epimorphisms of pointed sets are only stable under pullbacks along monomorphisms.

We now devote the rest of the section to a list of key properties of \mpnormal{} categories. From now on, and throughout the paper unless otherwise specified, we fix a category $\cat C$ with a (full, replete) mono-coreflective subcategory $\tobj$ such that $\cat C $ is \wpnormal{}. We denote by $\corefl{}$ the coreflector and by $\cuni{}$ the counit of the coreflection. Additional hypotheses will be introduced as needed.
\begin{proposition}
\label{fact-sys}
\facsys{} is a (orthogonal) factorisation system on $\cat C$. Moreover, if $\cat C$ is \normal{}, such factorisation system is stable.
\end{proposition}
\begin{proposition}
\label{ses-pb-po}
Let \(
\begin{tikzcd}[cramped, column sep =2em]
    A\arrow[r, "f",normmono]&B\arrow[r, "g",normepi]&C
\end{tikzcd}
\)
be a \npexact{} sequence in $\cat C$. Then the composite $\comp fg$ factors through an object $Z\in\tobj$, i.e.\ there exist morphisms $\unin\colon A\to Z$ and $\cunn\colon Z\to C$ such that the following diagram commutes. 
\begin{equation}
\zcsetup{reftype=diagram}
\begin{tikzcd}
\label{exact-sequence-square}
    A\arrow[r, "f"]\arrow[d, "\unin"']& B\arrow[d, "g"]
    \\
    Z\arrow[r, "\cunn"'] & C
\end{tikzcd}
\end{equation}
For any such factorisation, the following conditions are equivalent:
\begin{enumerate*}[(a)]
\zcsetup{reftype=condition}
    \item\label{ex-a}\zcref{exact-sequence-square} is both a pullback and a pushout; 
    \item \label{ex-b}$(Z,\cunn)$ is a $\tobj$-coreflection of $C$; 
    \item \label{ex-c}$(Z,\unin)$ is a $\tobj$-reflection of $A$. 
\end{enumerate*}
\end{proposition}
\begin{proposition}
\label{ses-morphism}
    In $\cat C$, consider the following commutative diagram with \npexact{} rows.
    \begin{equation*}
    \begin{tikzcd}
        A'\arrow[r,"{f'}",normmono]\arrow[d,"a"']&B'\arrow[r,"{g'}",normepi]\arrow[d,"b"']&C'\arrow[d,"c"']
        \\
        A\arrow[r, "f"',
        normmono]& B\arrow[r,"g"',normepi]& C
    \end{tikzcd}
    \end{equation*}
    The left-hand square is a pullback if and only if $c\in\TKer$.
\end{proposition}
\begin{proposition}
    \label{Noether}
    In $\cat C$, for any commutative diagram of the form shown below on the left, where $m$, $n$ (and therefore $j$) are \npnormmono{}s, we can construct the corresponding diagram on the right, where $p$, $q$ and $r$ are the \npcokernel{}s of $m$, $n$ and $j$ respectively, and $\phi$ and $\psi$ are the induced maps making the diagram commutative.
\begin{equation}
\zcsetup{reftype=diagram}
\label{Noet-iso-diagram}
\hskip\textwidth minus \textwidth
\begin{tikzcd}[baseline=(current bounding box.center)]
    & N\arrow[d, normmono, "n"]\arrow[dl, bend right, normmono, "j"']
    \\
    M\arrow[r, normmono, "m"'] & A
\end{tikzcd}
\hskip\textwidth minus \textwidth
\hskip\textwidth minus \textwidth
\begin{tikzcd}[baseline=(current bounding box.center)]
N\arrow[r,equal]\arrow[d, "j"',normmono] & N\arrow[d, normmono, "n"]
\\
M\arrow[r, normmono, "m"']\arrow[d,normepi, "r"] & A\arrow[d,normepi, "q"]\arrow[r,normepi, "p"] & A/M\arrow[d,equal]
\\
M/N\arrow[r, "\phi"'] & A/N\arrow[r,"\psi"'] & A/M
\end{tikzcd}
\hskip\textwidth minus \textwidth
\end{equation}
The diagram on the right in \ref{Noet-iso-diagram},  has \npexact{} rows, i.e.\ $M/N$ is a $\tobj$-normal subobject of  $A/N$ and $\frac{A/N}{M/N}\iso A/M$.
\end{proposition}
\begin{proposition}
    \label{pb-is-po}
Assume that $\cat C$ is \normal{} and that all \npnormepi{}s in $\cat C$ are regular epimorphisms. Then any pullback square of \npnormepi{}s is a pushout.
\end{proposition}

\section{Characterisations of torsion and torsion-free subcategories}
\label{sec-T-and-TF}
In this section, we provide some characterisations of \ztn{} and \ztf{} subcategories of the \wnormal{} category $\cat C$.

In the context of pointed categories, some of the most well-known characterisations of torsion-free subcategories are given in terms of particular pullback-stability conditions on the units of a normal-epi-reflective subcategory. Such conditions appear, in different contexts, in various forms and degrees of strength, such as semi-left-exactness, unit stability, reflection stability, normality of the adjoint (see for example \cite{BOURN,THOLEN,JANELIDZE07}). In the non-pointed setting, some of these characterisations fail to hold in their general form, and certain formulations that are equivalent in the pointed case cease to be equivalent. This breakdown is largely due to the fact that, outside the pointed context, 
the kernel of a pullback of a given morphism is not, in general, isomorphic to the kernel of the original morphism (see \zcref{kernel-of-parall,remark-pb}). Notwithstanding these limitations, we now provide two characterisations of \ztf{} subcategories that hold in the context of \wnormal{} categories, possibly under additional hypotheses. Pointed counterparts of these results in the homological context can be found in \cite[Lemma~4.10 and Theorem~4.12]{BOURN}.
See also \zcref{tt-not-sle} for an example illustrating the failure of these characterisations in more general formulations.
{
\begin{proposition}
\label{TF-stable}
 Let $\TFC$ be a \zrefl{} subcategory of $\cat C$, with $\TFP{}\colon \cat C\to\TFC$ the reflector and $\TFU{}$ the unit of the reflection. Then the following are equivalent.
 \begin{enumerate}[i)]
     \item \label{TF-stable-1} $\TFC$ is \ztf{}.
     \item\label{TF-stable-2} The units of the reflection $\TFP{}\colon \cat C\to\TFC$ are stable under pullbacks along \npnormmono{}s, meaning that for any pullback square
     \begin{equation}
     \label{square-stable-units}
     \begin{tikzcd}
     P\arrow[d, "p"']\arrow[r, "f'"] & Y\arrow[d, "\TFU Y",normepi]
     \\
     X\arrow[r,normmono, "f"'] &\TFP Y
     \end{tikzcd}
     \end{equation}
     with $f$ a \npnormmono{},  $(X,p)$ is an $\TFC$-reflection of $P$. 
 \end{enumerate}
\end{proposition}
\begin{proof}
Suppose that \zcref{TF-stable-1} holds, and let $\cat T$ be the subcategory of $\cat C$ such that $(\cat T,\TFC)$ is a \ztt.
Consider a pullback square of the form \ref{square-stable-units} and let  
\[
\begin{tikzcd}
    T_Y\arrow[r,normmono] & Y\arrow[r, normepi, "\TFU Y"] &\TFP Y
\end{tikzcd}
\]
be a \ttt{}-presentation of $Y$, with $T_Y\in \cat T$. As $p$ is obtained as the pullback along a \npnormmono{} of $\TFU Y$, we have that:
\begin{itemize}[-]
    \item $p$ is a \npnormepi{};
    \item by \zcref{kernel-of-parall}, there exists a map $T_Y\nmonoto P$ making $T_Y$  the \npkernel{} of $p$.
\end{itemize}
Thus, we obtain a \npexact{} sequence 
\begin{equation}
\label{t-tf-sequence-pb}
\begin{tikzcd}
   T_Y\arrow[r, normmono] & P\arrow[r,normepi, "p"] &X.
\end{tikzcd}
\end{equation}
Now, $f$ is a \npnormmono{} from $X$ to $\TFP Y\in\TFC$, and therefore $X\in\TFC$ by \zcref{other-TF-properties}. We conclude that the above sequence \ref{t-tf-sequence-pb} is a $(\cat T,\TFC)$-presentation of $P$, and hence $(X,p)$ is an $\TFC$-reflection of $P$ by \zcref{TF-is-ref}.

Suppose now that \zcref{TF-stable-2} holds. By \zcref{H-K}, we just need to prove that if $(X,k)$ is the \npkernel{} of $\TFU Y$ for some $Y$, then $\TFP X\in\tobj$. According to \zcref{monocoref}, there exists a map $h\colon X\to \corefl{\TFP X}$ such that the following square is a pullback. 
\[
\begin{tikzcd}   X\arrow[r,normmono, "k"]\arrow[d, "h"']& Y\arrow[d, "\TFU Y", normepi]
   \\
   \corefl{\TFP Y}\arrow[r,"\cuni {\TFP Y}"',normmono] &\TFP Y
\end{tikzcd}
\]
By \ref{TF-stable-2}, $(\corefl{\TFP Y},h)$ is an $\TFC$-reflection of $X$, and therefore $\TFP X\iso \corefl{\TFP Y}\in\tobj$.
 \end{proof}
}
{
\begin{proposition}
\label{natsq-pb}
    Suppose $\cat C$ is \normal{} (not just \wnormal{}) and that every \npnormepi{} in $\cat C$ is also a regular epimorphism.

 Let $\TFC$ be a \zrefl{} subcategory of $\cat C$,  with $\TFP{}\colon \cat C\to\TFC$ the reflector and $\TFU{}$ the unit of the reflection. Then the following are equivalent.
 \begin{enumerate}[i)]
     \item \label{natsq-pb-1} $\TFC$ is \ztf.
     \item \label{natsq-pb-2} For every object $X$ in $\cat C$, if $(K,k)$ is the \npkernel{} of $\TFU X$, then the following naturality square is a pullback.
     \begin{equation}
     \label{natsq-pb-d}
     \begin{tikzcd}
         K\arrow[r, "k",normmono]\arrow[d, "\TFU K"',normepi]& X\arrow[d, "\TFU X", normepi]
         \\
         \TFP K\arrow[r, "\TFP k"'] &\TFP X
     \end{tikzcd}
     \end{equation}
 \end{enumerate}
\end{proposition}
\begin{proof}
    If $\TFC$ is \ztf{}, then, in any naturality square like \ref{natsq-pb-d}, $K$ is a torsion object. Therefore, the map $\TFU K$ is the \npcokernel{} of the isomorphism between $K$ and its torsion part. Since the \npcokernel{} of an isomorphism is a $\tobj$-reflection of its domain, we conclude that $(\TFP K,\TFU K)$ is a $\tobj$-reflection of $K$. It follows from \zcref{ses-pb-po} that \ref{natsq-pb-d} is a pullback.

   Conversely, suppose that \ref{natsq-pb-2} holds. Fix an object $X\in\cat C$ and consider the pullback square \ref{natsq-pb-d}. Clearly $\mcomp{k,\TFU X}=\mcomp{\TFU K, \TFP k}$ is \nptrivial{}, and since $\TFU K$ is a \npnormepi{}, it follows that $\TFP k$ is \nptrivial{}. By coreflectivity of $\tobj$, there exists a map $u\colon \TFP K\to \corefl{\TFP X}$ such that $\mcomp{u,\cuni{\TFP X}}=\TFP k$. We thus obtain the following commutative diagram.
    \[
    \begin{tikzcd}
        K\arrow[r, equal]\arrow[d,"\TFU K"',normepi] & K\arrow[r, "k", normmono] \arrow[ d]& X\arrow[d,normepi, "\TFU X"]
        \\
        \TFP K\arrow[r, "u"'] & \corefl{\TFP X}\arrow[r,"\cuni{\TFP X}"']&\TFP X
    \end{tikzcd}
    \]
    The right-hand side is a pullback, since $k$ is the \npkernel{} of $\TFU X$. It follows that the left-hand square is a pullback and that the central vertical map $\mcomp{\TFU K, u}$ is a \npnormepi{}. By the general properties of factorisation systems, since both $\mcomp{\TFU K,u}$ and $\TFU K$ are \npnormepi{}s, we deduce that $u$ is one as well. The left-hand square is therefore a pullback diagram whose sides are all \npnormepi{}s, and it is thus a pushout (\zcref{pb-is-po}). We conclude that $u$ is an isomorphism and so $\TFP K\iso\corefl{\TFP X}\in\tobj$. This proves that $\TFC$ is \ztf{} by \zcref{H-K}.
\end{proof}
}
We conclude this section with a characterisation of \ztn{} subcategories. Its pointed counterpart is found in \cite[Theorem~4.6.2]{THOLEN}.
{
\begin{proposition}
\label{tor-charact}
    Let $\TNC$ be a \zcor{} subcategory of $\cat C$. Then the following are equivalent.
    \begin{enumerate}[i)]
        \item \label{tn-char-1} $\TNC$ is a \ztn{} subcategory;
        \item \label{tn-char-2}$\TNC$ is closed under \zext{}s.
    \end{enumerate}
\end{proposition}
\begin{proof}
We denote by $\TNF\colon\cat C\to\TNC$ the right adjoint to the inclusion $\TNC \to\cat C$, and by $\TNU{}$ the counit.
    The implication \ref{tn-char-1}$\implies$\ref{tn-char-2} follows from \zcref{other-TF-properties}.

    Vice versa, assume that \ref{tn-char-2} holds. By the dual of \zcref{H-K}, it suffices to prove that the \npcokernel{}s of the components of the counit $\TNU{}$ have \nptrivial{} $\TNC$-coreflections. Let $X$ be any object in $\cat C $ and let $(Q,q)$ be the \npcokernel{} of $\TNU X$. We aim to show that $\TNU Q$ is \nptrivial{}. Consider the following commutative diagram,
    
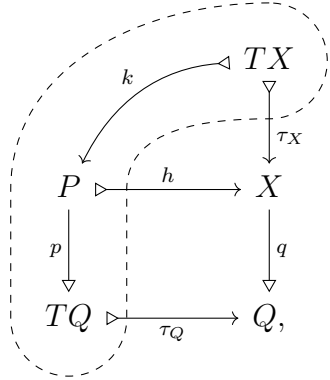
\[
    \begin{tikzcd}[%
    every matrix/.append style = {name=D}, column sep= 4em,
    execute at end picture={
    \node[minimum size =3.7em, name=1] at (D-1-2) {};
\node[minimum size =3.7em, name=2] at (D-2-1) {};
\node[minimum size =3.7em, name=3] at (D-3-1) {};
\draw[dashed] (1.north) to[out=180, in=90] (2.west) to (3.west) to[out=-90, in =180] (3.south) to[out=0,in=-90] (3.east) to[out=90, in =-90] (3.east) to (2.east) to[out=90, in = 180] (1.south) to[out=0, in= -90] (1.east) to[out=90, in=0] (1.north);
    }]
        &\TNP X\arrow[d,normmono, "\TNU X"{yshift=-.5em}]\arrow[dl, bend right, "k"', normmono]
        \\
        P\arrow[r, "h", normmono]\arrow[d,"p"', normepi] & X\arrow[d,normepi, "q"]
        \\
        \TNP Q\arrow[r, "\TNU Q"',normmono] & Q,
    \end{tikzcd}
\]
where the lower square is the pullback of $q$ along $\TNU Q$, while $k$ is the unique map such that the upper triangle commutes and $(\TNP X,k)$ is the \npkernel{} of $p$, as ensured by \zcref{kernel-of-parall}. Since $p$ is the pullback of a \npnormepi{}, it is itself a \npnormepi{}. Hence, the sequence encircled in the above diagram is \npexact{}. By closure of $\TNC$ under \zext{}s, it follows that $P\in\TNC$, and we can therefore write
\[
\mcomp{p,\TNU Q}=\mcomp{h,q}=\mcomp{\TNU P\inv,\TNP h,\TNU X,q},
\]
where the last equality follows from the naturality of $\TNU{}$ and the fact that $\TNU P$ is an isomorphism. The composite $\mcomp{\TNU X,q}$ is \nptrivial{}, and thus $\mcomp{p,\TNU Q}$ is \nptrivial{}. Since $p$ is a \npnormepi{}, it follows that $\TNU Q$ is \nptrivial{}, as desired.
\end{proof}
}
\section{Torsion theories and closure operators}
\label{sec-tt-and-CO}
There is a deep connection between torsion theories and closure operators. In a pointed category, under appropriate hypotheses, torsion theories are in a bijective correspondence with certain closure operators on the class of normal monomorphisms in that category (see \cite{BORCEUX94b,BOURN,THOLEN}). More recently, using the language of non-pointed torsion theories, it was shown that for any (reasonably well-behaved) class of monomorphisms on a category, a closure operator on that class is equivalent to a (non-pointed) torsion theory on a category having those monomorphisms as objects. In this section, we show that both of these correspondences can be seen as instances of a more general correspondence between \ztts{} on a \wnormal{} category and a special class of closure operators on \npnormmono{}s of that category.

Recall that $\cat C$ denotes a \wnormal{} category. Throughout this section, we will also assume that the class $\tobj$ of trivial objects in $\cat C$ is a reflective subcategory of $\cat C$. This ensures that every identity morphism is a \npnormmono{}. We will denote by $\refl{}$ the reflector and by $\uni{}$ the unit of the reflection. Notice that, since the components of the counit of the coreflection are monomorphisms, the components of the unit of the reflection are automatically epimorphisms and hence \npcokernel{}s (by the dual of \zcref{monocoref}; see also \cite{PN}).

With these assumptions in place, we now introduce the properties of the morphisms that will serve as the basis for the closure operators considered in this section.
\begin{definition}
Let $\monoclass$ be a class of monomorphisms in a category $\cat E$. A map in $\monoclass$ will be called an \emph{\mmono{}}. Assume $\cat E$ admits pullbacks along \mmono{}s. We say that $\monoclass$ is a \emph{\wscm{}} if the following properties hold.
\begin{enumerate}
    \item $\monoclass$ contains all identity morphisms;
    \item For all $f$ and $g$, composable maps in $\cat E$, if $g$ and $\comp fg$ are \mmono{}s, then $f$ is also an \mmono{}. 
    \item The pullback of an \mmono{} along any map in $\cat E$ is again an \mmono{} (in particular $\monoclass$ is closed under isomorphisms).
\end{enumerate}
 Given an object $X$ in $\cat E$, we will denote by $\monod X$ the class of \mmono{}s having $X$ as domain.

Given monomorphisms $m$ and $n$, we write $m\cont n$ if $m$ and $n$ have the same domain and there exists a (necessarily unique) map $j$ such that $\comp mj=n$. (Note that if $m,n\in\monoclass$, then $j\in\monoclass$.) If both $m\cont n$ and $n\cont m$, then we write $m\same n$ and the map $j$ is an isomorphism.
\end{definition}
(The term `semi-stable' is chosen for consistency with \cite{PN}, where both stable and semi-stable classes are discussed.)
\begin{remark}
    Due to the reflectivity of $\tobj$, the class of \npnormmono{}s in $\cat C$ is \wstab{} (see \cite{PN}).
\end{remark}
We now fix a set of definitions for various types of closure operators. While terminology and axiomatics vary across the literature, we adopt a set of names and conventions that suit our purposes here, without claiming any canonical status for these choices. We refer the reader for example to \cite{DIKRANJAN10} for an extensive guide to closure operators.
{

\begin{definition}
    Let $\monoclass$ be a \wscm{} on some category. A \emph{closure operator $\clo$ on $\monoclass$} is given by a family of functions $(\clf X\colon\monod X\to\monod X)_{X}$ such that the following axioms are satisfied for all $a,b\in\monod X$ and all $f\colon X\to Y$.
    \begin{enumerate}
        \item Extension: $a\cont \cl Xa$.
        \item Monotonicity: $a\cont b\implies \cl X a\cont\cl X b$.
        \item Continuity: $\cl Y {\preim f a}\cont \preim f{\cl Xa}$.
\end{enumerate}
        
        Given $a\in\monod X$, we say that $a$ is \emph{$C$-closed} (or simply \emph{closed}) if $\cl X a\same a$; we say that $a$ is \emph{$C$-dense} (or simply \emph{dense}) if $\cl Xa\same \id X$.

        If, additionally, the following condition holds for all $a\in\monod X$, then $\clo$ is called \emph{idempotent}.
        \begin{enumerate}
        \setcounter{enumi}{3}
            \item Idempotency: $\cl X{\cl Xa}\same \cl Xa$.
        \end{enumerate}
        
        Finally, we say that an idempotent closure operator $\clo$  is \emph{\wuniv{}} if, for all $a\in\monod X$, with $a\colon A\to X$, $\cl X a\colon B\to X$ and $j \colon A \to B$ the unique morphism such that $\mcomp{j,\cl Xa}=a$, the following additional condition is satisfied.
        \begin{enumerate}
        \setcounter{enumi}{4}
        \item 
            \Wunivty{}: $\cl B {j}\same \id B$.
        \end{enumerate}
        Given two closure operators $C$ and $C'$ on the same class of monomorphisms, we call them isomorphic (and write $C\same C'$) if for all objects $X$ and all $a\in\monod X$ one has $c_X(a)\same c'_X(a)$. We will often implicitly identify closure operators up to isomorphism.
\end{definition}

Next, we introduce some specialised classes of closure operators that are of interest in the context of a \wnormal{} category $\cat C$; these will be the ones involved in the characterisation of torsion theories.  
\begin{definition}
    A \emph{\nclosop{} on $\cat C$} is an idempotent closure operator on the class of \npnormmono{}s of $\cat C$ such that for all \npnormmono{}s $a\colon A\nmonoto X$ and all \npnormepi{}s $f\colon X\nepito Y $ one has
    $\cl Y {\preim f a}\same \preim f{\cl Xa}$.
\end{definition}
Of course this definition is based on that of \emph{homological closure operators} (\cite{BOURN}).

We are finally ready to state and prove the main result of this section.
{
\newcommand{\Qok}[1]{Q_{#1}}
\newcommand{\qok}[1]{q_{#1}}
\renewcommand{\TNF}{K}%
\renewcommand{\TNP}[1]{\TNF_{#1}}%
\renewcommand{\TNU}[1]{k_{#1}}
\begin{proposition}
\label{TT-and-CO}
    There is a bijective correspondence between \zrefl{} subcategories of $\cat C$ and \nclosop{}s on $\cat C$. This correspondence specialises to a bijective correspondence between \ztf{} subcategories of $\cat C$ and \wuniv{} \nclosop{}s on $\cat C$ in such a way that the \ztt{} corresponding to a closure $C$ operator is given by
\begin{equation}
\label{T-F-clo}
    \begin{split}
    \TNC &= \{ X\in\cat C\mid \cuni X\colon \corefl X\to X \text{ is $C$-dense}\},
    \\
 \TFC & = \{ X\in\cat C\mid \cuni X\colon \corefl X\to X \text{ is $C$-closed}\}
      \end{split}
\end{equation}
    (here, we recall that $\cuni X\colon\corefl X\nmonoto X$ denotes a $\tobj$-coreflection of $X$).
    \end{proposition}
    \begin{proof}
    The proof proceeds essentially as in \cite[Theorems~2.4 and 4.15]{BOURN}, with the initial maps $0_X\colon0\to X$ replaced by the $\tobj$-coreflections $\cuni X\colon \corefl X\nmonoto X$. For completeness, we briefly outline the main steps and highlight the key points where additional care is required, or where meaningful differences are involved.
    
    Given a closure operator $\clo=(\clf X)_{X\in\cat C}$, the associated subcategory takes the form $\TFC$ specified in \ref{T-F-clo}, with the reflection of an object $X$ given by the \npcokernel{} of $\cl X{\cuni X}$.

    Vice versa, let $\TFC$ be a \zrefl{} subcategory of $\cat C$, with reflector $\TFF\colon \cat C\to\TFC$ and unit $\TFU{}$. For any object $X$ in $\cat C$, let $\TNU X\colon \TNP X\nmonoto X$ be the \npkernel{} of $\TFU X$. Given a \npnormmono{} $a\colon A\nmonoto X$ in $\cat C$, we denote by $\qok a\colon X\nepito \Qok a$ its \npcokernel{}. The closure $\cl  X a$ of $a$ is then defined as the pullback of $\TNU {\Qok a}$ along $\qok a$. By \zcref{pb-nmono}, $\cl Xa$ can be equivalently defined as the \npkernel{} of $\mcomp{\qok A, \TFU {\Qok a}}$. We prove, for example, that for any \npnormmono{} $a\colon A\nmonoto X$ and any \npnormepi{} $f\colon Y\nepito X$, we have $\cl Y {\preim f a}\same \preim f{\cl Xa}$. As seen in the proof of \cite[Theorem~2.4]{BOURN}, this reduces to proving that in the following diagram,
    \[
    \begin{tikzcd}[column sep =4em ]
        P\arrow[r, "p",normmono]\arrow[d] & Y\arrow[d, "f", normepi] \arrow[r,normepi, "\qok p"]&\Qok p\arrow[d, "h", dotted]
        \\
        A\arrow[r, "a"', normmono] & X\arrow[r, normepi, "\qok a"'] &\Qok a
    \end{tikzcd}
    \]
    where the left-hand square is a pullback by construction, the induced map $h$ is an isomorphism. Indeed, $h$ has \nptrivial{} \npkernel{} by \zcref{ses-morphism}, and it is also a \npnormepi{}: $f$ and $\qok a$ are \npnormepi{}s, so their composite $\mcomp{f,\qok a}=\mcomp{\qok p,h}$ is a \npnormepi{}; since $\qok p$ is also a \npnormepi{}, it follows by the general properties of factorisation systems that $h$ is a \npnormepi{} as well. 

    For the second part of the proposition, let $\TFC $ be a \zrefl{} subcategory of $\cat C$, with reflector $\TFF \colon \cat C\to\TFC$ and unit $\TFU{}$, and let $\clo=(\clf X)_{X\in\cat C}$ be the associated closure operator. 
    For any object $X\in\cat C$, let $\TNU X\colon \TNP X\nmonoto X$ denote the \npkernel{} of $\TFU X$, so that by construction we have $\TNU{\TNP X}\same\cl {KX} {\cuni{\TNP X}}$. Clearly, the composite $\mcomp{\cuni X,\phi_X}$ is \nptrivial{}, and therefore there exists a unique morphism $j_X\colon \corefl X\nmonoto \TNP X$ such that $\mcomp{j_X,k_X}=\cuni X$. By coreflectivity of $\tobj$, we have that $j_X\cont \cuni{\TNP X}$. For the converse inequality $\cuni{\TNP X}\cont j_X$, consider the following diagram. 
    \[
    \begin{tikzcd}[column sep = 5em]
        \corefl{\TNP X}\arrow[r, "\corefl{\TNU X}"]\arrow[d, normmono, "\cuni{\TNP X}"'] & \corefl X\arrow[d, normmono, "\cuni X"]\arrow[dl, normmono, "j_X"description]
        \\
        \TNP X\arrow[r,normmono, "\TNU X"'] & X
    \end{tikzcd}
    \]
    The outer square commutes by naturality of $\cuni{}$, and the right triangle commutes by definition of $j_X$. Since $\TNU X$ is a monomorphisms, it follows that the left triangle also commutes. Hence $\cuni{\TNP X}\cont j_X$, and we conclude that $\cuni{\TNP X}\same j_X$. Therefore, we have
    \[
    \TNU{\TNP X}\same\cl {\TNP X}{\cuni {\TNP X}}\same\cl {\TNP X}{ j_X}.
    \]
    If $\clo$ is \wuniv{}, we also have $\cl {\TNP X}{ j_X}\same \id X$, proving that $\TNU{\TNP X}$ is an isomorphism for all objects $X$. It then follows from \zcref{idmpt-rad} that $\TFC$ is \ztf{}. Conversely, if $\TFC$ is torsion-free, then $\TNU{\TNP X}$ is an isomorphism for all $X$, and thus the equation above shows that $j_X$ is dense for all $X$. The argument then proceeds as in \cite[Theorem~4.15]{BOURN}.
    \end{proof}
}
{
\newcommand{\NS}[1]{{#1}^{\flat}}
\begin{remark}
\label{clos-op-on-mono}
    Let $\monoclass$ be a \wscm{} on a category $\cat E$ with pullbacks. We represent $\monoclass$-monomorphisms with the special arrow `$\smonoto$'. By \cite{PN}, if we denote by $\tobj$ and $\monocat$  the full subcategories of $\Arr E$ generated by the class of isomorphisms in $\cat E$ and by the class $\monoclass$, respectively, then $\monocat$ is a \wnormal{}  category. A morphism in $\monocat$ from $a\colon A\smonoto X$ to $b\colon B\smonoto Y$ is given by a pair of morphisms $(u,v)$ in $\cat E$ such that the following diagram commutes. 
    \[
    \begin{tikzcd}
         A\arrow[r, "u"]\arrow[d, "a"',smono] & B\arrow[d,"b",smono]
        \\
        X\arrow[r, "v"'] & Y
    \end{tikzcd}
    \]
   Such a morphism 
   \begin{enumerate*}[(1)]
       \item is a  \npnormmono{} if and only if $v$ is an isomorphism; 
       \item is a \npnormmono{} if and only if $u$ is an isomorphism and $v\in\monoclass$; 
       \item has \nptrivial{} \npkernel{} if and only if the diagram is a pullback in $\cat E$.
   \end{enumerate*} In particular, every \npnormepi{} is, up to isomorphism, of the canonical form $(j,\id X)\colon a\nepito b$, with $a,b,j\in\monoclass$ as in \zcref{subobj} below, and every \npnormmono{} is, up to isomorphism, of the canonical form $(\id A,b)\colon j\nmonoto a$, with $a,b,j\in\monoclass$ as in \zcref{subobj} below.
   \begin{equation}
   \zcsetup{reftype=diagram}
   \label{subobj}
   \begin{tikzcd}
       A\arrow[r, "j",smono]\arrow[dr,smono, bend right, "a"']&B\arrow[d,smono,"b"]
       \\
       &X
   \end{tikzcd}\quad\implies\quad
   \begin{tikzcd}
       A\arrow[r, equal]\dar[smono, "j"']\arrow[dr,phantom, "\scriptstyle\tobj\textnormal{-normal}"{sloped,yshift=.4em}, "\scriptstyle\textnormal{mono}"'{sloped,yshift=-.4em}] & A\dar[smono, "a"]
       \\
       B\rar["b"',smono] & X
   \end{tikzcd},\quad
   \begin{tikzcd}
       A\rar["j",smono]\dar[smono, "a"']\arrow[dr,phantom, "\scriptstyle\tobj\textnormal{-normal}"{sloped,yshift=.4em}, "\scriptstyle\textnormal{epi}"'{sloped,yshift=-.4em}] & B\dar[smono, "b"]
       \\
       X\rar[equal] & X
   \end{tikzcd}
   \end{equation}

    There is a bijective correspondence between the following classes.
    \begin{enumerate}
        \item \label{S-clos-op} \Wuniv{} closure operators on $\monoclass$;
         \item\label{Z-S-clos-op} \wuniv{} \nclosop{}s on $\monocat$;
        \item \label{S-TT} \ztts{} on $\monocat$.
    \end{enumerate}
    The correspondence between \ref*{S-clos-op} and \ref*{S-TT} is a well-known result (see \cite{GRANDIS20}), and, in particular, the torsion theory associated to a closure operator $\clo$ on $\monoclass$ is given by 
    \begin{equation*}
    \begin{split}
        &\cat T = \{a\in\monoclass\mid a \textnormal{ is } \clo\textnormal{-dense}\},
        \\
        &\cat F = \{a\in\monoclass\mid a \textnormal{ is } \clo\textnormal{-closed}\}.
    \end{split}
    \end{equation*}
    The correspondence between \ref*{Z-S-clos-op} and \ref*{S-TT} is instead due to \zcref{TT-and-CO}. We want to explicitly describe the correspondence between \ref{S-clos-op} and \ref{Z-S-clos-op}.
    
\newcommand{\clos}{C^\sharp}%
\newcommand{\clfs}[1]{c^\sharp_{#1}}%
\newcommand{\cls}[2]{c^\sharp_{#1}(#2)}
    Given a closure operator $\clo=(\clf X)_{X\in\cat E}$, we define the closure operator $\clos=(\clfs a)_{a\in\monocat}$, which acts on \npnormmono{}s. For a \npnormmono{} $(\id A,b)\colon u\to a$ in $\monocat$, with $a,b,j\in\monoclass$ as in \zcref{subobj}, we set $\cls a{(\id A,b)}=(\id A, \cl X b)\colon \comp uj\to a$, where $j$ denotes the unique map such that $\mcomp{j,\cl X b}=b$.
    \[
    \begin{tikzcd}
        A\arrow[r,equal]\arrow[d, "u"',smono] & A\arrow[r, equal]\arrow[d] & A\arrow[d, "a",smono]
        \\
        B\arrow[r, "j"',smono]\arrow[rr, bend right, "b"', smono,yshift =-.5em, shorten = -.5em]&\cdot\arrow[r, "\cl Xb"',smono] & X
    \end{tikzcd}
    \]
   If $\clo$ is idempotent, then $\clos$ is a \nclosop{} on $\monocat$, and if $\clo$ is \wuniv{}, then so is $\clos$.

\newcommand{\dlo}{D}
\newcommand{\dlf}[1]{d_{#1}}
\newcommand{\dl}[2]{d_{#1}({#2})}
\newcommand{\dlob}{D^{\flat}}
\newcommand{\dlfb}[1]{d_{#1}^\flat}
\newcommand{\dlb}[2]{d_{#1}^\flat({#2})}
    Conversely, given a closure operator $\dlo=(\dlf a)_{a\in\monocat}$ on \npnormmono{}s, we define $\dlob=(\dlfb X)_{X\in \cat E}$ as follows: for any $b\colon B\smonoto X$ in $\monoclass$, suppose $\dl b{(\id B, b)\colon \id B\nmonoto b}=(\id B,d)$, with $d\in\monoclass$; then we set $\dlb Xb=d$. (Here, we assume that $\dlo$ preserves the `standard form' $(\id{},s\in\monoclass)$ of \npnormmono{}s, which holds up to isomorphism of closure operators.) We have that $\dlob$ is a closure operator on $\monoclass$; moreover, if $\dlo$ is idempotent or \wuniv{}, then so is $\dlob$.

    Clearly, for any closure operator $\clo$ on $\monoclass$, we have $\clo^{\sharp\flat}=C$. Vice versa, for a \nclosop{} $\dlo$ on $\monocat$, we have $\dlo^{\flat\sharp}\same \dlo$. This equivalence relies on the observation that, for any \npnormmono{} $(\id A,b)\colon j\to a$ in $\monocat$, with $a,b,j\in\monoclass$ as in \zcref{subobj}, we have 
    \[
    \begin{multlined}[0.67\textwidth]
    \dlf a (\id A,b)=\dlf a \Bigl((j,\id X)^{-1}\bigl((\id B,b)\colon \id B\nmonoto b\bigr)\Bigr)\same\\\same(j,\id X)^{-1}\Bigl(\dlf b\bigl((\id B,b)\colon \id B\nmonoto b\bigr)\Bigr)=\dlf a^{\sharp\flat}(\id A,b),\end{multlined}\]
    since $(j,\id X)\colon a\nepito b$ is a \npnormepi{}.
\end{remark}
}
}
\section{Torsion theories and factorisation systems}{
\newcommand{\lc}{\mathscr E}
\newcommand{\rc}{\mathscr M}%
\label{sec-TT-and-FS}
The connection between torsion theories, reflective subcategories and factorisation systems has been widely studied (\cite{CASSIDY,ROSICKY,EVERAERT,GRANDIS20,HTT}). It is a well-known result that any reflective subcategory of a given category gives rise to a prefactorisation system, whose left class consists of the morphisms inverted by the reflector (see \cite{CASSIDY}). If the category is finitely complete and the reflection is semi-left-exact (in the sense of \cite{CASSIDY}), this prefactorisation system becomes a genuine factorisation system.

In the pointed case, under suitable hypotheses on the ambient category, the reflection associated with a torsion theory is always semi-left-exact, and thus induces a factorisation system (\cite{EVERAERT,THOLEN}). This is not the case in the non-pointed setting. Indeed, as shown in \zcref{no-semi-left}, even under \pnormality{} assumptions, torsion-free subcategories are not semi-left-exact, and hence the results of \cite{CASSIDY} do not automatically apply.

However, in \cite{EVERAERT}, it was shown that in a homological category, there is an alternative way to associate a factorisation system to a torsion theory. This approach yields a correspondence between torsion theories satisfying a certain property (denoted `(N)' in \cite{EVERAERT}) and stable factorisation systems on the category.

In this section, we aim to partially recover this result for the non-pointed \wnormal{} category $\cat C$. As a first step, we introduce the notion of \emph{\car{}} \ztt. This is a non-pointed version of property (N) in \cite{EVERAERT}.
\begin{definition}
We say that a \ztt{} is \emph{\car} if for every object $X$, the torsion part $\TNU X\colon \TNP X\nmonoto X$ of $X$ is a characteristic \npnormmono{} in $X$, i.e.\ for any \npnormmono{} $k\colon K\nmonoto \TNP X$ we have that the composite $\mcomp{k,\TNU X}\colon K\to X$ is again a \npnormmono{}. 
\end{definition}
The term `characteristic' is borrowed from group theory, as characteristic subgroups can be described using the property given in the above definition. Notice that all torsion theories in the category of groups are (0-)characteristic. 

We are now ready to state the main result of this section.
In the remainder of the section, we assume -- just as in \zcref{sec-tt-and-CO} -- that the subcategory $\tobj$ of trivial objects in the \wnormal{} category $\cat C$ is reflective, and keep the same notation $\refl{}$ and $\uni{}$ for the reflector and the unit of the reflection, respectively.
\begin{proposition}
    There is a bijective correspondence between \car{} \ztts{} on $\cat C$ and factorisation systems $(\lc,\rc)$ on $\cat C$ such that $\lc$ is a class of \npnormepi{}s closed under pullbacks along \npnormmono{}s.
\end{proposition}
\begin{proof}
    The proof follows the same outline of the one of \cite[Proposition~3.5]{EVERAERT}, with initial and terminal maps replaced by  $\tobj$-coreflections and $\tobj$-reflections, respectively. Some additional care is required when multiple $\tobj$-coreflections and $\tobj$-reflections are involved simultaneously. We outline the main steps of the proof.

    Given a \car{} \ztt{} $(\TNC,\TFC)$, the associated factorisation system is given by
    \[
    \begin{split}
        \lc&= \{f\in \NEpi\mid\zker(f)\in\TNC\},
        \\
        \rc &= \{f\in\ar(\cat C)\mid \zker (f)\in\TFC\},
    \end{split}
    \]
    where we denoted by `$\tobj\operatorname{-ker}$' the object part of \npkernel{}s.
    Following \cite{EVERAERT} without substantial complications, showing that these two classes of morphisms define a factorisation system satisfying the desired properties reduces to applying the non-pointed version of Noether's third isomorphism theorem (\ref{Noether}) and the fact that  the pullback of a morphism along a \npnormmono{} has the same \npkernel{} as the original morphism (\zcref{kernel-of-parall}).

    Conversely, given a factorisation system $(\lc,\rc)$ on $\cat C$ with $\lc\subseteq \NEpi$ and closed under pullbacks along \npnormmono{}s, the associated torsion theory is defined by
    \[
    \begin{split}
        \TNC&= \{X\in \cat C\mid\uni X\in \lc\},
        \\
        \TFC &= \{X\in\cat C\mid \uni X\in \rc\},
    \end{split}
    \]
    and the $(\TNC,\TFC)$-presentation of an object X is given by the encircled sequence in the following diagram,
   \[
   \begin{tikzcd}[%
    every matrix/.append style = {name=D}, 
    execute at end picture={
    \node[minimum size =3.4em, name=1] at (D-1-1) {};
    \node[minimum size =3.4em, name=2] at (D-1-3) {};
    \draw[dashed] (1.north) to[out=0, in=180] (2.north) to[out=0, in=90] (2.east) to[out=270, in=0] (2.south) to[out=180, in =0] (1.south) to[out=180,in=270] (1.west) to[out=90, in=180] (1.north);
}]
       \TNP X\arrow[r,"\TNU X", normmono] & X\arrow[r,normepi, "\TFU X"]\arrow[dr, bend right, "\uni X"', normepi] & \TFP X\arrow[d, "\psi_X"{yshift=-.4em}]
       \\
       && \refl X
   \end{tikzcd}
   \]
   where $\uni X=\mcomp{\TFU X, \psi_X}$ is the $(\lc,\rc)$-factorisation of $\uni X$, and $\TNU X$ is the \npkernel{} of $\TFU X$. Once again, the proof mostly follows the one in \cite{EVERAERT}. As a representative case -- one that involves some subtleties -- we verify that this \ztt{} is \car{}. Let $k\colon K\nmonoto X $ a \npnormmono{}, and, unlike in \cite{EVERAERT}, we will need to consider its \npcokernel{} $f\colon X\nepito Y$. This choice ensures, by \zcref{ses-pb-po}, that the $\tobj$-reflection  of $K$ and the $\tobj$-coreflection of $Y$ are isomorphic, and that there exists a map $\cuni{}\colon\refl K\to Y$ such that in the following diagram the outer rectangle is a pullback (and a pushout). In the same diagram we also take the $(\lc,\rc)$-factorisations $f=\comp em$ and $\uni K=\mcomp{e',m'}$, and we denote by $d$ the induced morphism.
 \[
   \begin{tikzcd}
       K\arrow[r,"e'",normepi]\arrow[rr, bend left, "\uni K",normepi]\arrow[d,normmono, "k"] & I'\arrow[r,"m'"]\arrow[d, dashed, "d"] & \refl K\arrow[d, "\cuni{}",normmono]
       \\
       X\arrow[r, normepi, "e"']\arrow[rr, bend right, "f"', normepi,bend right] & I\arrow[r, "m"'] & Y
   \end{tikzcd}
   \]
   The uniqueness of the $(\lc,\rc)$-factorisation of $\uni K$, together with the pullback-stability of both $\lc$ and $\rc$, ensures that the right-hand square is a pullback and hence $d$ is a \npnormmono{}. Since the \npkernel{} of $e'$ is $\TNU K$, it follows by \zcref{kernel-of-parall} that the \npkernel{} of $e$ is $\mcomp{\TNU K,k}$.
\end{proof}
\begin{remark}
    Let $(\lc,\rc)$ be the factorisation system associated to a \car{} \ztt{} \ttt{} on $\cat C$. Denote by $F\colon \cat C\to \TFC$ the reflector and by $\TFU{}\colon \id{\cat C}\tto \TFF$ the unit of the reflection. We claim that the functor $\TFF$ inverts all maps in $\lc$, and hence  $\lc$ is contained in the left class of the `classical' prefactorisation system associated with the reflective subcategory $\TFC$ (see the beginning of this section and \cite{CASSIDY}) 
    
    Let $f\colon X\nepito Y$ be any map in $\lc$, and let $a\colon A\nmonoto X$ be its \npkernel{}. By construction, $A\in\TNC$. We show that the map $Ff\colon FX\to FY$ is the \npcokernel{} of $Fa$ in $\TFC$. 
    
    First recall that $\tobj\subseteq\TFC$. Now, let $b\colon \TFP X\to B$ be a map in $\TFC$ such that $\mcomp{Fa,b}$ is \nptrivial{}. Then the composite $\mcomp{a,\TFU X, b}=\mcomp{\TFU A, \TFP a,b}$ is also \nptrivial{}. Therefore, there exists a unique map $u\colon Y\to B$ such that $\comp f u=\mcomp{\TFU X,b}$. By reflectivity, there exists a unique $v\colon \TFP Y\to B$ such that $u=\mcomp{\TFU Y,v}$.
    This $v$ clearly  satisfies $\mcomp{\TFP f, v}=b$. Finally observe that $\TFP f$ is an epimorphism since $f$ is and $\TFF$ is left-adjoint. This concludes the proof.
    \[
    \begin{tikzcd}[column sep =4em, bezier bounding box]
        A\rar[normmono, "a"]\dar[normepi, "\TFU A"] &X\rar[normepi, "f"] \dar[normepi, "\TFU X"] &Y\dar[normepi, "\TFU Y"]\arrow[ddl,controls={+(7em,-2em) and +(8em,-1em)}, "u"near start, dashed, start anchor=east]
        \\
        \TFP A\rar["\TFP a"]&\TFP X\rar["\TFP f"] \dar["b"']&\TFP Y\arrow[dl,"v",curve={height=-.4em},end anchor=north east, start anchor={[xshift=.5em]},dashed]
        \\
        &B
    \end{tikzcd}
    \]
    Since $A\in\TNC$, it follows that $\TFP A\in\tobj$, and thus $\TFP a$ is \nptrivial{}. As a result, $\TFP f$ is an isomorphism, since it is the \npcokernel{} of the \nptrivial{} map $\TFP a$.
\end{remark}
}
\section{Hereditary torsion theories}
\label{sec-hereditary}
{
In this section we focus on characterisations of \emph{hereditary} torsion theories. Naturally, in the context of \normal{} categories, the appropriate notion of heredity is not grounded in monomorphisms, but rather in the right part of the factorisation system relevant to this setting -- namely, morphisms with a \nptrivial{} \npkernel{}. We therefore give the following definition 
\begin{definition}
A \ztt{} $(\TNC,\TFC)$ on $\cat C$ is \emph{\zher{}} if for every morphism $X\tkerto Y$ with \nptrivial{} \npkernel{} such that $Y\in\TNC$, the object $X$ is also in $\TNC$.
In this case, we say that $\TFC$ is a \zher{} \ztf{} subcategory of $\cat C$.
\end{definition}
We can also recover a characterisation of \zher{} \ztf{} subcategories from \cite{BOURN} with a similar proof.
\begin{proposition}
\label{hered-charact}
    Let $\TFC$ be a \ztf{} subcategory of $\cat C$, and let $\TFP{}\colon \cat C\to \TFC$ be the left adjoin to the inclusion $\TFC\hookrightarrow\cat C$. Then $\TFC$ is \zher{} if and only if $\TFP{}$ preserves maps with \nptrivial{} \npkernel{}.
\end{proposition}
\begin{proof}
    Let $\TNC$ be the \ztn{} subcategory associated to $\TFC$. Denote by $\TFU{}$ the unit of the reflection of $\cat C$ into $\TFC$, and, for every object $X\in\cat C$, let $\TNU X\colon \TNP X\to X$ be the \npkernel{} of $\TFU X$. 
    
    Suppose $\TFC$ is \zher{}. For any morphism $f\colon X\tkerto Y$ with \nptrivial{} \npkernel{}, consider the following diagram,
    \[
    \begin{tikzcd}
        K\arrow[r,normmono, "k"] \arrow[d, "f_0"]& X \arrow[r,normepi, "q"]\arrow[d, tker, "f"]& Q\arrow[d, "f_1"]
        \\
        \TNP Y\arrow[r,normmono, "\TNU Y"'] & Y\arrow[r, "\TFU Y"',normepi] &\TFP Y,
    \end{tikzcd}
    \]
    where the left-hand square is a pullback, $(Q,q)$ is the \npcokernel{} of $k$ and $f_1$ is the induced map. By \zcref{ses-morphism}, $f_1$ has \nptrivial{} \npkernel{}, and hence $Q\in\TFC$ by the general properties of \ztf{} subcategories (see \zcref{other-TF-properties}). On the other hand, $f_0$ also has \nptrivial{} \npkernel{} because it is a pullback of $f$, and thus $K\in\TNC{}$ by the \zherty{} hypothesis. By \zcref{TF-is-ref}, we  conclude that the upper sequence in the above diagram is uniquely isomorphic to the canonical $(\TNC,\TFC)$-presentation \(
    \begin{tikzcd}[cramped, column sep =2em]
        \TNP X\arrow[r,normmono, "\TNU X"] & X\arrow[r, "\TFU X",normepi] &\TFP X
    \end{tikzcd}
    \)
    of $X$, and, in particular $\TFP f\iso f_1$ has \nptrivial{} \npkernel{}.

    Conversely, suppose $\TFP{}$ preserves maps with \nptrivial{} \npkernel{}. We know an object $Y$ lies in $\TNC$ if and only if $\TFU Y$ is \nptrivial{} (see \zcref{char-tor-obj}). Let $f\colon X\tkerto Y$ be a map in $\TKer$, with $Y\in\TNC$. Then the map $\mcomp{f,\TFU Y}=\mcomp{\TFU X, \TFP f}$ is \nptrivial{} because $\TFU Y$ is. Since $\TFP f$ has \nptrivial{} \npkernel{}, we conclude that $\TFU X$ is \nptrivial{} (use \zcref{tker-charact}) and thus $X\in\TNC$.
\end{proof}
The results  from \zcref{sec-tt-and-CO} about \nclosop{} and \ztt{} can be specialised to \zher{} \ztt{} via the following notion (see \cite{BORCEUX94a,BOURN}).
\begin{definition}
   Let $\monoclass$ be a \wscm{} on some category. An idempotent closure operator $\clo$ on $\monoclass$ is said to be \emph{hereditary} if the following condition holds for all $a\colon A\monoto X$ in $\monoclass$ and all arrows $f\colon Y\to X$.
   \begin{enumerate}
   \setcounter{enumi}{5}
       \item Heredity: $\cl Y {\preim f a}\same\preim f{\cl X a}$.
   \end{enumerate}
\end{definition}
\begin{remark}
    Hereditary closure operators are always weakly hereditary, as shown in \cite{BOURN}.
\end{remark}
We then have the following characterisation of \zher{} \ztts{}.
\begin{proposition}
    The bijective correspondence of \zcref{TT-and-CO} specialises to a bijective correspondence between \zher{} \ztts{} on $\cat C$ and hereditary \nclosop{}s on $\cat C$.
\end{proposition}
\begin{proof}
    The proof of this proposition follows the one of \cite[ Proposition 5.4]{BOURN}, with monomorphisms replaced with maps with \nptrivial{} \npkernel{}, and making use of the characterisation in \zcref{hered-charact}. The relevant properties of monomorphisms in homological categories are substituted, in the context of \wnormal{} categories, by the following properties of maps with \nptrivial{} \npkernel{}:
    \begin{itemize}[-]
        \item they are the right class of a factorisation system, with \npnormepi{}s on the left;
        \item they may be detected in morphisms of \npexact{} sequences as in \zcref{ses-morphism};
        \item for any morphism $f\colon X\tkerto Y$ with \nptrivial{} \npkernel{}, the canonical square induced by the $\tobj$-coreflection
        \[
         \begin{tikzcd}
         \corefl X\arrow[r, "\corefl f"] \arrow[d, "\cuni X"',normmono]&\corefl Y\arrow[d, "\cuni Y",normmono]
         \\
         X\arrow[r, "f"', tker]& Y
         \end{tikzcd}
        \]
        is a pullback (by \zcref{monocoref,tker-charact}).\qedhere
    \end{itemize}
\end{proof}
\begin{remark}
    The bijective correspondence of \zcref{clos-op-on-mono} can be specialised to a bijective correspondence between the following classes.
    \begin{enumerate}
        \item\label{HcloM} Hereditary closure operators on $\monoclass$;
        \item hereditary \nclosop{}s on $\monocat$;
        \item\label{HttM} \zher{} \ztts{} on $\monocat$.
    \end{enumerate}
    We just prove the correspondence between \ref*{HcloM} and \ref*{HttM}, the rest follows from \zcref{clos-op-on-mono}.

    Given a hereditary closure operator $\clo$ on $\cat M$, the corresponding \ztn{} subcategory of $\monocat$ is given by the $\clo$-dense monomorphisms of $\monoclass$. If $(u,v)\colon a\tkerto b$ is a morphism in $\monocat$ with \nptrivial{} \npkernel{}, with $a\colon A\smonoto X$ and $b\colon B\smonoto Y$ in $\monoclass$ and $b$ being $\clo$-dense, then we have that $a\same\preim u b$ by the characterisation of maps with \nptrivial{} \npkernel{} recalled in \zcref{clos-op-on-mono}. Therefore we obtain
    \[
    \cl X a\same \cl X{\preim u a}\same\preim u{\cl Y b }\same \preim u {\id Y}\same\id X,
    \]
    i.e.\ $a$ is $\clo$-dense.

    Vice versa, let $(\cat T,\cat F)$ be a \zher{} \ztt{} and let $\clo$ be the corresponding closure operator on $\monoclass$, so that $\cat T$ is the class of $\clo$-dense monomorphisms in $\monoclass$. Let $a\colon A\smonoto X$ be in $\monoclass$ and let $f\colon Y\to X$ be any morphism in $\cat E$. We aim to show that $\preim f{\cl Xa}\cont \cl Y{\preim f a}$. Consider the following diagram where both squares are pullbacks.
    \[
    \begin{tikzcd}[column sep=6em]
        \cdot\dar \arrow[r, "j",smono]\arrow[rr, bend left, "\preim f a",smono, yshift=.2em]&P\dar\arrow[r, "\preim f {\cl Xa}",smono] &Y\dar["f"]
        \\
        A \arrow[r, "i"',smono]\arrow[rr, bend right, "a"',smono, yshift=-.1em]&\cdot\arrow[r, "\cl Xa"',smono] &X
    \end{tikzcd}
    \]
Since the left-hand square is a pullback, it determines a morphims $j\tkerto i$ with \nptrivial{} \npkernel{}, and since $i$ is $\clo$-dense by weak heredity, it follows that $j$ is also $\clo$-dense by \zherty{}. We also have that 
\[
\mcomp{\cl P j, \preim f{\cl X A}}\cont \cl Y{\preim f a}.
\]
This follows from the general observation that for any commutative diagram 
\[
\begin{tikzcd}[ row sep =1.7em, column sep = 5em]
    M\arrow[r, "m_1",smono] \arrow[dr, "m_2"', bend right,smono] &X_1
   \\
   &X_2\arrow[u, "x"',smono]
\end{tikzcd}
\]
with $m_1,m_2,s\in\monoclass$, one has that the closure of $M$ in the `smaller space' $X_2$ is smaller than the closure of $M$ in the `larger space' $X_1$, that is, $\mcomp{\cl {X_2}{m_2}, x}\le \cl {X_1}{m_1} $; notice however that $\mcomp{\cl {X_2}{m_2}, x}$ might not be in $\monoclass$. Now, since $\cl P j$ is an isomorphism, the claim follows.
\end{remark}
}
\section{Examples}
\label{sec-examples}
\label{no-semi-left}
In this section, we examine examples of (non-pointed) torsion theories (in the sense of \zcref{ztt}) within (non-pointed) \wpnormal{} and \pnormal{} categories. Interestingly, several torsion theories already present in the literature happen to be situated within (semi-)prenormal categories, and hence automatically enjoy all the properties established in this paper; we list some of these torsion theories in the following table, along with the corresponding references, and refer the reader to \cite{PN} for details on the \wpnormal{}-category structure.
{
\begin{center}
    \begin{tblr}{colspec={|Q[5.8em,valign=m,halign=c]|Q[5.8em,valign=m,halign=c]|Q[5.8em,valign=m,halign=c]|Q[5.8em,valign=m,halign=c]|Q[valign=m,halign=c]|}, row{1}={gray!10!white}}
    \hline
        Category & $\tobj$ & $\tobj$-torsion part &$\tobj$-tor.-free part&Ref.
        \\\hline
        Abelian, semi-abelian, homological or normal cats. &0&\SetCell[c=2]{c}{Any pointed torsion theory} 
        && /
        \\\hline
        Comm. monoids &$0$&Abelian groups &Reduced comm. mons.&\cite{FACCHINI21}
        \\\hline
        Preordered groups &0&Discrete preorders &Partial orders&\cite{MICHEL}
        \\\hline
        Preordered groups &Discrete preorders  &Equiv. relations &Partial orders &\cite{MICHEL}
        \\\hline
        Embeddings of top.\ spaces & \hspace*{0pt}Homeomorphisms & Closed embeddings & Dense embeddings& \cite{GRANDIS20}
        \\\hline
        Monos of top.\ spaces  & \hspace*{0pt} Homeomorphisms & Embeddings & Bijective cont.\ maps & \cite{BORCEUX94a}
        \\\hline
    \end{tblr}
\end{center}
}
\begin{remark}
\label{tt-not-sle}
The examples in the last two rows are both special instances of the correspondence described in \zcref{sec-tt-and-CO} and in \cite{GRANDIS20}; they correspond, respectively, to the standard topological closure of a subspace, and to the closure operator that assigns to an injective continuous function $a\colon A\monoto X$ the inclusion of the subspace generated by the image of $a$ in $X$. In particular, this last closure operator is hereditary (`universal' in \cite{BORCEUX94a}), and hence the associated \ztt{}  is \zher{}. Moreover, in the second-to-last example, note that the torsion-free subcategory of closed subspaces is not semi-left-exact, meaning that  reflections are not necessarily stable along maps in the subcategory. For instance, consider the following pullback in the category in question.
\[\begin{matrix}
    \begin{tikzcd}
    i\dar\rar &a\dar
    \\
    b\rar &\id X
\end{tikzcd}
\end{matrix}
\quad\quad \begin{pmatrix}\iff \begin{tikzcd}[ampersand replacement =\&]
    \varnothing\dar[hookrightarrow]\rar[hookrightarrow] \&A\dar[hookrightarrow]
    \\
    B\rar[hookrightarrow] \& X
\end{tikzcd}\end{pmatrix}
\]
Here, $X$ is the closed unit interval $[0,1]$ with its usual topology; $i\colon\varnothing\to X$ is the initial map; $b\colon B\to X$ is the embedding of the singleton $B=\{0\}\subseteq X$; $a\colon A\to X$ is the embedding of the open unit interval $A=(0,1)\subseteq X$; all maps in the pullback are induced by the identity on $X$. Clearly $a\to\id X$ is the reflection of $a$ into the subcategory of closed embeddings (since the closure of $A$ is all of $X$), and both $b$ and $\id X$ lie in this subcategory, as $X$ and $B$ are closed subspaces of $X$. However $i\to b$ is not the reflection of $i$ (which is instead given by $\id i$).
\end{remark}
Finally, we devote the rest of the section to introducing new examples of torsion theories in various pointed and non-pointed \wpnormal{} categories.
\subsection*{Torsion theories for commutative $\Omega$-monoids}
\newcommand{\forom}{U_\Omega}
For a given algebraic theory $\Omega=(\Sigma,\Tau)$, consisting of a set $\Sigma$ of operations of finite positive arities and a set $T$ of equational axioms, a commutative $\Omega$-monoid $M$ is a commutative monoid equipped with an $\Omega$-algebra structure such that for every $n$-ary operation $\sigma$ in $\Omega$, the function $\sigma\colon M^n\to M$ is monoid morphism in each variable separately. The category $\catOm$ of $\Omega$-monoids is prenormal, as shown in \cite{PN}, where one also finds explicit characterisations of normal epimorphisms. A normal monomorphism in $\catOm$ is, up to isomorphism, the inclusion of  a \emph{normal $\Omega$-submonoid} $A\inclusion M$, that is an $\Omega$-submonoid having the following closure properties: \begin{enumerate*}
        \item for all $x,y\in M$, if $x,x+y\in A$, then $y\in A$;
        \item for every $n$-ary operation $\sigma\in\Sigma$ and every $x=(x_1,\dots,x_n)\in M^n$, $\sigma(x)\in A$ whenever one of the $x_i\in A$.
\end{enumerate*}
We show that, under suitable hypotheses, torsion theories on commutative monoids can induce torsion theories on commutative $\Omega$-monoids. 
\begin{proposition}
\label{tts-on-om-mon}
    Let \ttt{} be a torsion theory on the category $\catCMon$ of commutative monoids. Let $\TNP{}$ denote the coreflector into the subcategory of torsion objects, and let $\TNU{}$ denote the counit of the coreflection. 
    
    Consider the obvious forgetful functor $\forom\colon \catOm\to\catCMon$ and assume that, for all commutative $\Omega$-monoids $M$, the normal submonoid $\TNP{\forom M}\subseteq\forom M$ is also a normal $\Omega$-submonoid of $M$. 
    
    Then the full subcategories of $\catOm$
   \[
   \TNC_\Omega =\{M\in\catOm\,\vert\, \forom M\in\TNC\}\hspace{.8em}\textnormal{and}\hspace{.8em}
        \TFC_\Omega =\{M\in\catOm\,\vert\, \forom M\in\TFC\}
   \]
   form a torsion theory on $\catOm$.
\end{proposition}
\begin{proof}
    This can be easily verified directly or using \zcref{tor-charact}. In any case, the key points are that $\TNP{}$ and $\TNU{}$ induce a normal-mono-coreflection of $\catOm$ into $\TNC_\Omega$, and that the relevant exact sequences in $\catOm$ have the same form as in $\catCMon$.
\end{proof}
\begin{example}
\label{TF-om-mon}
    One can easily prove that the following torsion theories on $\catCMon$ satisfy the hypotheses of \zcref{tts-on-om-mon} for any $\Omega$.
    \begin{enumerate}
        \item \emph{Torsion and torsion-free commutative monoids.} An element $x$ of a commutative is torsion if $nx=0$ for some positive integer $n$. A commutative monoid is torsion (respectively, torsion-free) if all (respectively, none) of its non-zero elements are torsion. For any commutative monoid $M$, one can consider the submonoid $\TNP M$ of its torsion elements. The cokernel $\TFU M\colon M\to\TFP M$ of $\TNP M\inclusion M$ is simply the quotient of $M$ by the relation  
        \[
        x\sim y \textnormal{ if } x+a=y+b \textnormal{ for some }a,b\in\TNP M.
        \]
        Clearly $\TNP M$ is torsion and $\TFP M$ is torsion-free.
        \item \emph{Abelian groups and reduced commutative monoids}. Abelian groups can be viewed as commutative monoids where every element is invertible. A commutative monoid is reduced (or pure, see \cite{Messora}) if its only invertible element is its unit. Abelian groups and reduced commutative monoids form a torsion theory on $\catCMon$ (\cite{FACCHINI21}).
    \end{enumerate}
\end{example}
\subsection*{A torsion theory for inverse monoids}
\newcommand{\catICM}{\cats{ICMon}}
\newcommand{\catECM}{\cats{ECMon}}
A commutative monoid $M$ is called an \emph{ inverse commutative monoid} if for every $x\in M$ there exists an element $x\inv\in M$, called the \emph{inverse} of $x$, such that
\begin{equation*}
\label{inv-axioms}
    xx\inv x=x \quad\textnormal{and}\quad x\inv x x\inv=x\inv.
\end{equation*}
(in this subsection we shall use multiplicative notation). The category $\catICM$ of inverse commutative monoids is (pointed) prenormal -- since it is in particular a category of $\Omega$-monoids -- but it is also prenormal with respect to the subcategory $\tobj=\catECM$ of idempotent commutative monoids, i.e., commutative monoids where every element is idempotent (see \cite{PN}). There is a functor $E\colon \catICM\to \catECM$ sending an inverse commutative monoid $M$ to the subset $E(M)$ of its idempotent elements, and acting on morphisms by restriction to idempotents. This functor is both left and right adjoint to the inclusion $\catECM\inclusion\catICM$. Recall from \cite{PN} that a  morphism $f\colon M\to N$ of inverse commutative monoids is a \npnormepi{} if and only if it is surjective and $E(f)$ is an isomorphism.

Just as the category $\catICM$ is prenormal with respect to two subcategories of trivial objects -- namely, the subcategory of zero monoids and the subcategory of idempotents commutative monoids -- we can now exhibit a subcategory of $\catICM$ which is torsion-free with respect to both.
\begin{proposition}
\label{TF-inv-mon}
Consider the full subcategory $\TFC$ of $\catICM$ consisting \emph{torsion-free inverse monoids}, that is, commutative inverse monoids $M$ such that for all $x\in M$, if $x^n=1$ for some positive integer $n$, then $x=1$ (please note again the use of multiplicative notation). The subcategory $\TFC$ is both (0-)torsion-free and \ztf{} in $\catICM$ (where $\tobj=\catECM$).
\end{proposition}
\begin{proof}
Since $\catICM$ is a category of $\Omega$-monoids, we already know from \zcref{TF-om-mon} (where \emph{additive} notation was used) that $\cat F$ is (0-)torsion-free in $\catICM$. We have also explicitly described the $\TFC$-reflection $\TFU M\colon M\to\TFP M$ for any $M$. As $\TFU{}$ is the unit of the reflection of a \emph{pointed} torsion-free subcategory, we know that it is stable under pullbacks along \emph{arbitrary} morphisms in $\TFC$. Hence, by \zcref{TF-stable}, to prove that $\TFC$ is also \ztf{} it suffices to show that the components of $\TFU{}$ are \npnormepi{}s. This reduces to showing that $E(\TFU M)$ is an isomorphism for all inverse commutative monoids $M$. Since $E$ is a left-adjoint, it preserve regular epimorphisms. Thus, it only remains to show that $E(\TFU M)$ is a monomorphism. Suppose $\TFU M(x)=\TFU M(y)$ for some $x,y\in E(M)$. By the description of $\TFU{}$ given in \ref{TF-om-mon}, we have that $xa=yb$ for some $a,b\in M$ such that $a^n=b^n=1$ for a positive integer $n$. Since $x$ and $y$ are idempotents, we have
\[
x=x^n=x^n1=x^na^n=y^nb^n=y^n1=y,
\]
completing the proof.
\end{proof}

\subsection*{Torsion theories for groupoids}
\newcommand{\grpd}[1]{\mathbb {#1}}
\newcommand{\Ker}[1]{\mathbb K_{#1}}
\newcommand{\vgrp}{\cat X}
\newcommand{\vgrpd}{\cats G_{\cat X}}
\newcommand{\vtor}{\cat T}
\newcommand{\vtord}{\cats G^0_{\vtor}}
The category of groupoids is \pnormal{} with respect to the subcategory $\tobj$ of discrete groupoids (\cite{GRANDIS13,PN}). We recall that if $\grpd G$ is a groupoid, a normal subgroupoid $\grpd N$ of $\grpd G$ is a subgroupoid which is both wide ($\grpd N$ contains all the objects of $\grpd G$) and closed under conjugation (if $h\in\grpd N$ and $g\in\grpd G$, then $\mcomp{g,h,g\inv}$ is in $\grpd N$ whenever it is defined). Every \npnormmono{} is, up to isomorphism, an inclusion of a normal subgroupoid. The kernel of a functor $F\colon \grpd G\to\grpd H$ -- which we denote by $\Ker F$ -- is the wide subgroupoid consisting of those morphisms in $\grpd G$  whose image under $F$ is an identity morphism. A functor $F\colon \grpd G\to\grpd H$ is a \npnormepi{} if and only if it is strictly surjective on arrows and whenever $F(g)=f(g')$ for some $g,g'\in\grpd G$, it follows that $\mcomp{g,u}=\mcomp{u',g'}$ for some $u,u'\in\Ker F$.

A groupoid $\grpd A$ is \emph{abelian} if for all objects $X\in\grpd A$, the automorphism group $\grpd A(X,X)$ is abelian. In particular, abelian groupoids, in this sense, are precisely the abelian objects of the category of groupoids (\cite{BOURN02}). Note that every discrete groupoid is abelian, and hence one can consider \npkernel{}s, \npcokernel{}s,  \npnormmono{}s and \npnormepi{}s in the category of abelian groupoids and functors. One can easily verify that these are characterised in the same way as in the category of groupoids.

For the remainder of this subsection, let $\vgrp$ denote either the category of groups or the category of abelian groups. Accordingly, we write $\vgrpd$ for the category of groupoids if $\vgrp$ is the category of groups, and for the category of abelian groupoids when $\vgrp$ is the category of abelian groups. We have the following proposition.
\begin{proposition}
    Let $\cat T$ be a torsion subcategory of $\vgrp$, and let $\vtord$ be the full subcategory of $\vgrpd$ consisting of those groupoids $\grpd G$ such that:
    \begin{itemize}
        \item for all distinct objects $X,Y\in\grpd G$, the hom-set $\grpd G(X,Y)$ is empty; 
        \item for every object $X\in\grpd G$, the endomorphism group $\grpd G(X,X)$ lies in $\vtor$.
    \end{itemize}
    Then $\vtord$ is a \ztn{} subcategory of $\vgrpd$. In particular, every torsion theory on the category of (abelian) groups yields a \ztt{} on the category of (abelian) groupoids.
\end{proposition}
\begin{proof}
We use \zcref{tor-charact}. First, we prove that $\vtord$ is \zcor{} in $\vgrpd$. Let $T\colon\vgrp\to \vtor$ be a right adjoint to the inclusion $\vtor \inclusion \vgrp$, so that $TG$ is a normal subgroup of $G$ for all $G\in\vgrp$. For all $\grpd G\in \vgrpd$, we define the wide subgroupoid $\grpd T\grpd G$ of $\grpd G$, whose hom-sets are given by
\[\grpd T\grpd G(X,Y)=
\begin{cases*}
   \varnothing & if $X\neq Y$,
   \\
   T\bigl(\grpd G(X,X)\bigr) & if $X=Y$,
\end{cases*}
\]
for $X,Y\in\grpd G$. We have that $\grpd T\grpd G\inclusion \grpd G$ is a $\tobj$-normal-mono-coreflection of $\grpd G$ in $\vtord$. This directly follows from the fact that $T\bigl(\grpd G(X,X)\bigr)\inclusion\grpd G(X,X)$ is a normal-mono-coreflection of $\grpd G(X,X)$ in $\vtor$ for all $X\in\grpd G$.

Next, we verify that $\vtord$ is closed under \zext{}s. Consider a \npexact{} sequence
\[
\begin{tikzcd}
    \grpd N\rar[hookrightarrow] &\grpd G \rar{F} &\grpd H,
\end{tikzcd}
\]
where $\grpd N,\grpd H\in\vtord$. Let $g\colon X\to Y$ be a morphism in $\grpd G$. 
Since all morphisms in $\grpd H$ are endomorphisms, we have $FX=FY$. 
By the characterisation of \npnormepi{}s recalled above, there exist $u,v\in\Ker F$ such that $\mcomp{\id X,u}=\mcomp{v,\id Y}$. 
As all morphisms in $\Ker F=\grpd N$ are endomorphisms, it follows that $X=\dom u=\cod v =Y$, proving that all morphisms in $\grpd G$ are endomorphisms. 
We also deduce that for all objects $X\in\grpd G$ the sequence
\[
\begin{tikzcd}
    \grpd N(X,X)\rar[hookrightarrow] &\grpd G(X,X)\rar{F} &\grpd H(FX,FX)
\end{tikzcd}
\]
is exact in $\vgrp$. Since $\grpd N(X,X)$ and $\grpd H(FX,FX)$ lie in $\vtor$, which is a torsion subcategory, we conclude that $\grpd G(X,X)\in\vtor$, and hence $\grpd G\in\vtord$.
\end{proof}
\subsection*{Torsion theories for preordered commutative monoids}
\newcommand{\conn}{\sim}
We briefly recall that a preordered commutative monoid is a commutative monoid equipped with a preorder that is compatible with addition. Morphisms between such objects are monotone monoid homomorphisms. The category of preordered commutative monoids is \pnormal{} (see \cite{PN}). In this category, a normal monomorphism is, up to isomorphism, given by the inclusion of a full, normal preordered submonoid, that is, a full preordered submonoid $K\subseteq M$ such that for all $x,y\in M$, if $x\in K$ and $x+y\in K$, then $y\in M$. Normal epimorphisms are instead characterised as maps $f\colon M\to N$ such that:
\begin{enumerate}
    \item for all $x,x'\in M$, if $f(x)=f(x')$, then there exist $a,a'\in\ker f$ such that $x+a=x'+a'$;
    \item for all $y\le y'$ in $N$, there exist $x\le x'$ in M such that $f(x)=y$ and $f(x')=y'$.
\end{enumerate}
Given a preordered commutative monoid $M$, two elements $x,x'\in M$ are said to be \emph{connected}, written $x\conn x'$, if there exists a finite zigzag of inequalities connecting them, as in the following diagram,
\[
    x=x_0\le x_1\ge x_2\le\dots\ge x_n=x',
\]
with $x_0, x_1,\dots,x_n\in M$. The relation $\conn$ is a congruence on the commutative monoid $M$, i.e.\ an equivalence relation satisfying 
\[  
x\conn x'\implies x+y\conn x'+y
\]
for all $x,x',y\in M$.
A preordered commutative monoid is said to be \emph{connected} if any two of its elements are connected. Connected preordered commutative monoids are part of a torsion theory.
\begin{proposition}
    The category of connected preordered commutative monoids is a torsion subcategory of the category of preordered commutative monoids. The corresponding torsion-free subcategory consists of those preordered commutative monoids whose  connected component containing 0 is trivial.
\end{proposition}
\begin{proof}
We use \zcref{tor-charact}.
Given any preordered commutative monoid $M$, define $TM \subseteq M$ to be the full preordered submonoid consisting of all elements connected to 0. 
This is clearly connected and $TM\inclusion M$ is the coreflection of $M$ into the subcategory of connected preordered commutative monoids.

We now show that $TM$ is a normal submonoid of $M$. Indeed, suppose $x, x + y \in TM$. Then $x,x+y\conn 0$, and since $\conn$ is a congruence, it follows that
\[
y=0+y\conn x+y\conn 0,
\]
and so $y\in TM$.

Next, consider a short exact sequence of preordered commutative monoids.
\[
\begin{tikzcd}
    L\rar[hookrightarrow] &M\rar{f} &N,
\end{tikzcd}
\]
with $L$ and $N$ connected. Consider $x\in M$. Since $N$ is connected, we have $f(x)\conn 0$ in $N$. By the characterisation of normal epimorphisms seen above, the zigzag exhibiting $f(x)\conn 0$ is reflected in $M$, that is, there exist $x',a$ in $M$ such that $x'\conn a$, $f(x')=f(x)$ and $f(a)=0$. 
Again because $f$ is a normal epimorphism, we obtain $b,b'\in L$ such that 
\[
x+b=x'+b'.
\]
Since $L$ is connected, the elements $a,b,b'$ are all connected to $0$ in $L$. Using repeatedly this fact and the fact that $\conn$ is a monoid congruence we deduce:
\[
x=x+0\conn x+b=x'+b'\conn x'+0=x'\conn a\conn 0.
\]
This proves that all elements in $M$ are connected to 0, and hence $M$ is connected.
\end{proof}
\begin{remark}
It is easy to show that connected partially ordered commutative monoids also form a torsion subcategory of the category of partially ordered commutative monoids. This follows  from the above proposition and the description of normal epimorphisms from \cite{PN}.
\end{remark}
\section*{Acknowledgments}
\addcontentsline{toc}{section}{Acknowledgments}
This research was conducted while both authors were affiliated with INdAM -- Istituto Nazionale di Alta Matematica ‘Francesco Severi’, Gruppo Nazionale per le Strutture Algebriche, Geometriche e le loro Applicazioni (GNSAGA). Moreover, the second author was partially supported by the HUMATH project (FIS-2023-04053 -- CUP I53C24003170001). 
\phantomsection
\addcontentsline{toc}{section}{References}
\printbibliography
\end{document}